\documentclass[12pt]{amsart}
\usepackage{amsmath,amssymb,txfonts}
\usepackage{times}
\usepackage{bbm}
\usepackage{amssymb}
\usepackage[dvips]{graphicx}
\usepackage{citeref}
\allowdisplaybreaks
\usepackage{amsmath,amssymb,txfonts}
\usepackage{amssymb}
\usepackage{amsxtra}
\usepackage{amsmath}
\usepackage{txfonts}
\usepackage{amsmath,amssymb,txfonts}
\usepackage{amssymb}
\usepackage{amsxtra}
\usepackage{amsmath}
\usepackage{txfonts}
\usepackage{xcolor}

\usepackage[cp1252]{inputenc}

\usepackage{mathrsfs}
\usepackage{graphicx}

\usepackage[active]{srcltx}

\usepackage{graphicx}

\numberwithin{equation}{section}

\newtheorem{theorem}{Theorem}[section]
\newtheorem{remark}[theorem]{Remark}

\newtheorem{lemma}[theorem]{Lemma}

\newcommand{\ct}[1]{\langle {#1}\rangle \lower.3ex\mathrm{$_{t}$}}
\newcommand{\lt}[1]{[ {#1}] \lower.3ex\mathrm{$_{t}$}}

\def\dint{\displaystyle\int}
\def\dsum{\displaystyle\sum}

\numberwithin{equation}{section}
\def\dfrac{\displaystyle\frac}

\begin{document}

\title[Quantitative weighted bounds for Calder\'{o}n commutator with rough kernel]
	{ Quantitative weighted bounds for Calder\'{o}n commutator\\ with rough kernel}
	
	\author{Yanping Chen\,}
	\address{School of Mathematics and Physics, University of Science and Technology Beijing, Beijing 100083, China}
	\email{yanpingch@ustb.edu.cn}

	\author{Ji Li}
\address{Department of Mathematics,
		Macquarie University, NSW, 2109, Australia}
	\email{ji.li@mq.edu.au}

	\thanks{The project was in part supported by: Yanping Chen's
		NSF of China (\# 11871096, \# 11471033); Ji Li's  DP 170101060.}

	
	\date{\today}
\vskip 1cm

\begin{abstract}
		We consider weighted $L^p(w)$ boundedness ($1<p<\infty $ and $w$ a Muckenhoupt $A_p$ weight) of the Calder\'{o}n commutator $\mathcal C_\Omega$ associated with rough homogeneous kernel,
		under the condition $\Omega\in L^q(\mathbb S^{n-1})$ for $q_0<q\leq\infty$ with $q_0$ a fixed constant depending on $w$. Comparing to the previous related known results (assuming $\Omega\in L^\infty(\mathbb S^{n-1})$), our result for $\Omega\in L^q(\mathbb S^{n-1})$ with $q$ in the range $(q_0,\infty)$ is new. We also obtain a quantitative weighted bound for this $\mathcal C_\Omega$ on 
		$L^p(w)$, which is the best known quantitative result for this class of operators.
	\end{abstract}
	
	\maketitle

%%%%%%%%%%%%%%%%%%%%%%%%%%%%  ?????? ????%%%%%%%%%%%%%%%%%%%%%
%%%%%%%%%%%%%%%%%%%%%%%%%%%%  ?????? ????%%%%%%%%%%%%%%%%%%%%%

 \pagestyle{plain} %???????
 \renewcommand{\thepage}{\arabic{page}}% ?????????????????
 \setcounter{page}{1} %?????1???
 \setcounter{equation}{0}%?????1???????

\section{Introduction}

\subsection{Background and statement of main reuslt}

 The Calder\'{o}n commutators (see \cite{C1, C2}) originate from a representation of linear differential operators by means of singular integral operators, which is an approach to the uniqueness of the Cauchy problem for partial differential equations (see \cite{C3}).

The first version was introduced by Calder\'{o}n \cite{C2}  %the first Calder\'{o}n commutator
$$ \Big[b, H\frac{d}{dx}\Big]f(x):=\text{p.v.}\int_{-\infty}^\infty\left(\frac{-1}{x-y}\right)\left(\frac{b(x)-b(y)}{x-y}\right)f (y)\,dy.$$

 It also plays an important role in the theory of  Cauchy integral along Lipschitz curve in $\mathbb C$ and the Kato square root problem on $\mathbb R$ (see \cite{C3, Fe, Me, MC} for the details).

A more general version is the Calder\'{o}n commutators with rough kernels
\begin{align}\label{point pv of C}
\mathcal{C}_\Omega f(x)=\lim_{\epsilon\to0^+}\mathcal{C}_\epsilon f(x), \quad a.e. \ x\in\mathbb R^n
\end{align}
defined initially for $f\in C_0^\infty(\mathbb R^n)$, where $\mathcal{C}_\epsilon f$ is the truncated Calder\'{o}n commutator of $f$:
 \begin{equation}\label{Cal}
\mathcal{C}_\epsilon f
(x):=\int_{|x-y|>\epsilon}\left(\cfrac{\Omega(x-y)}{|x-y|^{n+1}}\right)\big(b(x)-b(y)\big)f(y)dy,\ \ \ \ \forall\ \ x\in\mathbb R^n,
 \end{equation}
  where  $\Omega$ is
homogeneous of degree zero, integrable on $\mathbb S^{n-1}$ (the unit sphere in $\mathbb R^n$) and satisfies the cancellation condition on the unit sphere
\begin{equation}\label{mean zero}
\int_{\mathbb S^{n-1}}\Omega(x')(x'_k)^N\,d\sigma(x')=0,\quad \ \forall\ \  (k,N)\in \{1,\dots, n\}\times\{0,1\}.
\end{equation}

Using the method of rotation, Calder\'{o}n \cite{C1} proved the boundedness of the commutator $\mathcal{C}_\Omega$ {\ for $\Omega$ in $L\log L(\mathbb S^{n-1})$ and $b\in Lip(\mathbb R^n)$}, and then obtained the boundedness of the operators $[b,T]\nabla$ and $\nabla[b,T]$, where $T$ is a homogeneous singular integral operator with some symbol $K$ which can be defined similarly as $\mathcal{C}_\Omega$, that is,
\begin{equation}\label{SIO}
\begin{cases}
Tf(x)=\lim_{\epsilon\to0^+}T_\epsilon f(x), \ a.e. \ x\in\mathbb R^n;\\[5pt]
T_\epsilon
f(x)=\int_{|y|>\epsilon}K(y)f(x-y)dy,
\end{cases}
 \end{equation}
where $f\in C_0^\infty(\mathbb R^n)$, and the kernel $K$ is homogeneous of degree $-n$, belongs to  $ L_{\rm loc}^1({\mathbb R}^n)$ and
 enjoys the cancellation on the unit sphere
\begin{equation}\label{can of O}
\int_{\mathbb S^{n-1}}K(y')d\sigma(y')=0.
 \end{equation}
Later on, many authors made important progress on the Calder\'on commutators, one can consult \cite{Co, CM, CM1,  CM2, Mu, SH, Y, Ta3, MuWh71, GH, CDH1} and the references therein for its development and applications.  Among these numerous references, we point out that Hofmann \cite{SH} first obtained the weighted estimate for $C_\Omega$ with $\Omega\in L^\infty(\mathbb S^{n-1})$ satisfying the cancellation condition \eqref{mean zero} (For the definition of Muckenhoupt $A_p$ weight, we refer the readers to Section 2).

It is not clear that to obtain the weighted $L^p(w)$ boundedness $(1<p<\infty, w\in A_p)$, whether the condition $\Omega\in L^\infty(\mathbb S^{n-1})$ can be weakened. 
Moreover, there is no known result for the  quantitative weighted bounds for Calderon commutator $\mathcal C_\Omega$ with rough kernel.

\smallskip
The main result of this paper is address these points. We prove that the condition on $\Omega$ can be reduced to $\Omega\in L^q(\mathbb S^{n-1})$ for $q_0<q\leq\infty$ with $q_0$ a fixed constant depending on $w$. We provide a quantitative weighted bounds for  $\mathcal{C}_\Omega$. We do not know whether this is sharp, but it is the best known quantitative result for this class of operators.   %(see \cite{HRT}).
\begin{theorem}\label{thm1}  Let
	$1<p<\infty,$  $w\in A_p,$
	 and $b\in Lip(\mathbb R^n).$
	 Suppose that $\mathcal{C}_\Omega$ with
	  $\Omega\in L^q(\mathbb S^{n-1}),\,q>\frac{1+\varepsilon}{\varepsilon}$ with $\varepsilon:=\frac{1}{2}\frac{c_n}{(w)_{A_p}}$ and satisfy \eqref{mean zero}.
	  Then we obtain that
\begin{equation}\label{N12w}
\ \|\mathcal{C}_\Omega f\|_{L^p(w)}\lesssim \|\Omega\|_{L^q}
 \{w\}_{A_p} (w)_{A_p} \|\nabla b\|_{L^\infty}\|f\|_{L^p(w)},
\end{equation}
where $(w)_{A_p}:=\max\Big\{[w]_{A_\infty},\,[w^{1-p'}]_{A_\infty}\Big\}$, $\{w\}_{A_p}:=[w]_{A_p}^{1\over p}\max\Big\{[w]_{A_\infty}^{1\over p'},\,[w^{1-p'}]_{A_\infty}^{1\over p}\Big\}$, and  $[w]_{A_r}$ ($1<r\leq\infty$) is the norm of $w$, given in \eqref{[Ap]}, and the implicit constant is independent of $b,f$ and $w$. In particular, we get $$ \|\mathcal{C}_\Omega f\|_{L^2(w)}\lesssim \|\Omega\|_{L^q}[w]_{A_2}^2 \|\nabla b\|_{L^\infty}\|f\|_{L^2(w)}.
$$
\end{theorem}

%{\color{blue} \bf TO BE CONTINUE}
To compare with the previous closely related results, we point out that the  sharp quantitative weighted bounds for singular integral with rough kernels have been studied intensively in the last three years with the key tool sparse domination (pointwise version originated in \cite{Lac15}).
Let $T_{\Omega}$ be the homogeneous singular integral operator defined by
\begin{align}\label{ts}
T_{\Omega}
f(x)=p.v.\int_{{\Bbb R}^n}\frac{\Omega(y')}{|y|^n}f(x-y)dy,
\end{align}
where $\Omega$ is
homogeneous of degree zero, integrable on $\mathbb S^{n-1}$  and
 satisfies the cancelation condition
$
\int_{\mathbb S^{n-1}}\Omega(y')d\sigma(y')=0.
$
Among these sharp quantitative weighted bounds for $T_\Omega$, we would like to highlight that
Hyt\"{o}nen--Roncal--Tapiola  \cite{HRT} first proved that when $\Omega\in L^\infty(\mathbb S^{n-1})$,
  \begin{align}\label{T}
  \|T_\Omega\|_{L^2(w)\rightarrow L^2(w)}\le C_n\|\Omega\|_{L^\infty}[w]_{A_2}^2.
  \end{align}
They introduced a two-step technique involving pointwise sparse domination  for Dini-type Calder\'on--Zygmund kernels, a Littlewood--Paley decomposition along the lines of \cite{DR86} and interpolation with change of measure from \cite{sw}.
 Later,
Conde-Alonso--Culiuc--Di Plinio--Ou \cite{CCDO} proved a sparse domination for the bilinear forms associated with $T_\Omega$ with $\Omega \in L^q(\mathbb S^{n-1})$ for $1<q\leq \infty$ satisfying the cancellation conditions, which leads to  quantitative weighted bounds for $T_\Omega$ with $\Omega \in L^\infty(\mathbb S^{n-1})$ (see (1.5) in Corollary A.1, \cite{CCDO}), extending previous result of \cite{HRT}.
Lerner \cite{Ler 0} provided another different approach to get the sparse domination in \cite{CCDO} via showing a weak type estimate for $T_\Omega$ with $\Omega \in L^\infty(\mathbb S^{n-1})$.

Recently \eqref{T} was also extended to maximal singular integrals $T_\Omega^\ast$ by Di Plinio, Hyt\"{o}nen and Li \cite{DHL} and Lerner \cite{Ler 2} via sparse domination, which gives
  \begin{align}\label{Tmax}
  \|T_\Omega^*\|_{L^2(w)\rightarrow L^2(w)}\le C_n\|\Omega\|_{L^\infty}[w]_{A_2}^2.
  \end{align}

We point out that by taking into account of our 
auxiliary result on interpolation, to obtain the $L^p(w)$-boundedness of $T_\Omega$, the kernel condition in \cite{HRT} and \cite{CCDO} can be reduced to our condition as in Theorem \ref{thm1}. 

Moreover, our result Theorem \ref{thm1}  is the first one to study this quantitative weighted bounds for Calder\'{o}n commutator $\mathcal C_\Omega$ with rough kernel.

\subsection{Approach and techniques}

{

To prove Theorem \ref{thm1}, we borrow the idea from \cite{HRT} via using a two-step approach involving the know result for Dini-type Calder\'on--Zygmund kernels. However, their techniques on decomposition of $T_\Omega$ is not applicable to Calder\'on commutator $C_\Omega$ since the symbol $b \in Lip(\mathbb R^n)$ is also involving in $\mathcal C_\Omega$. 
Moreover, to weaken the assumption on $\Omega$, we need to provide new ingredients.  To overcome this problem, we provide the following methods:
\begin{itemize}

\item Littlewood--Paley decomposition:

We writie $\mathcal{C}_\Omega f=[b, T_{\Omega}^1]f=b(x)T_{\Omega}^1f(x)-T_{\Omega}^1(bf)(x)$,
with
$T_{\Omega}^1=\sum_{k \in \mathbb{Z}} T_{k} f$, where $T_{k} f= K_{k}\ast f$ with  $ K_{k}=\frac{\Omega(x^{\prime})}{|x|^{n+1}} \chi_{\{2^{k}<|x|\le2^{k+1}\}}$. And then by constructing the Littlewood--Paley decomposition of each $T_k$  as
$$T_k
 =\sum_{j=1}^{\infty}T_k \Delta_{k,N(j)} +\sum_{j=1}^{\infty}T_k\widetilde \Delta_{k,N(j)}=:\sum_{j=1}^{\infty} T_{1,j}^{N} + \sum_{j=1}^{\infty} T_{2,j}^{N},
$$
where $ \Delta_{k,N(j)}
=\sum_{i=N(j-1)+1}^{N(j)}\Delta_{k-i}^3$ and $ \widetilde\Delta_{k,N(j)}
=\sum_{i=-N(j)-1}^{-N(j-1)}\Delta_{k-i}^3$,
with $\Delta_i$ the convolution operator formed by a Schwartz function $\psi$ whose Fourier transform has compact support $(1/2,2)$, and $N(j)$ is the jump introduced in \cite{HRT}, we get
\begin{align*}
\mathcal{C}_\Omega f=\sum_{i=1}^{2} \sum_{j=1}^{\infty} [b, T_{i,j}^{N}] f .
\end{align*}

\item Revised interpolation:

Assume that $\varepsilon\in (0,1)$ and $1\leq p,\, q_0\leq\infty$, that $w$ is a positive weight, and that $T^\Omega$ is a sublinear operator associated with $\Omega$ satisfying
$$
\|T^\Omega f\|_{L^p}\leq k_0\|\Omega\|_{L^{q_0}({\Bbb S}^{n-1})}\|f\|_{L^p},
$$ and
$$
\|T^\Omega f\|_{L^p(w^{1+\varepsilon})}\leq k_1\|\Omega\|_{L^\infty({\Bbb S}^{n-1})}\| f\|_{L^p(w^{1+\varepsilon})}.
$$
 Then
$$
\|T^\Omega f\|_{L^p(w)}\leq k_0^{\frac{1}{1+\varepsilon}} k_1^{\frac{\varepsilon}{1+\varepsilon}}\|\Omega\|_{L^q({\Bbb S}^{n-1})}\| f\|_{L^p(w)},
$$
 where
$
\frac{1}{q}=\frac{\varepsilon}{1+\varepsilon}\frac{1}{q_0}.$

\item unweighted boundedness and quantitative weighted boundedness for $[b, T_{i,j}^{N}]$:

\textbf{Part 1. } We do not assume the cancellation condition \eqref{mean zero} on $\Omega.$ From Lemma \ref{sw1}, it suffices to build the following two versions of boundedness for $[b, T_{1,j}^{N}]:$
\begin{align*}
& \|[b, T_{1,\,j}^{N}]f\|_{L^p}\lesssim \|\nabla b\|_{L^\infty}\|\Omega\|_{L^{q_0}}2^{-\theta N(j-1)}\|f\|_{L^p}, \ \ q_0>1,\,  \theta\in(0,1),\\
 &  \|[b, T_{1,\,j}^{N}]f\|_{L^p(w)}\lesssim \|\nabla b\|_{L^\infty}\|\Omega\|_{L^{\infty}}(1+N(j)) \{w\}_{A_p}\|f\|_{L^p(w)},
 \end{align*}
which follows from proving that for every $b\in Lip(\mathbb R^n)$, for each fixed $j\in\mathbb N$, and $N(j)$, $[b, T_{1,\,j}^{N}]$
is a Dini-type Calder\'on--Zygmund operator associated with the Dini function $\omega_{1,j}(t)=\|\Omega\|_{L^\infty}\|\nabla b\|_{L^\infty}\min\{1,2^{N(j)}t\}$.

\textbf{Part 2.}  Suppose that $\Omega$ satisfies the cancellation condition \eqref{mean zero}, we only need $\Omega\in L^1({\Bbb S}^{n-1}),$ which is very rough. From Stein--Weiss \cite{sw} interpolation with change of measure (Lemma \ref{sw}), it suffices to build the following two versions of boundedness for $[b, T_{2,j}^{N}]:$
\begin{align*}
& \|[b, T_{2,\,j}^{N}]f\|_{L^2}\lesssim \|\nabla b\|_{L^\infty}\|\Omega\|_{L^1}2^{-\theta' N(j-1)}(1+N(j))\|f\|_{L^2},  \theta'\in(0,1),\\
 &  \|[b, T_{2,\,j}^{N}]f\|_{L^p(w)}\lesssim \|\nabla b\|_{L^\infty}\|\Omega\|_{L^1}(1+N(j)) \{w\}_{A_p}\|f\|_{L^p(w)},
 \end{align*}
which follows from proving that for every $b\in Lip(\mathbb R^n)$, for each fixed $j\in\mathbb N$, and $N(j)$, $[b, T_{2,\,j}^{N}]$
is a Dini-type Calder\'on--Zygmund operator associated with the Dini function $\omega_{2,j}(t)=\|\Omega\|_{L^1}\|\nabla b\|_{L^\infty}\min\{1,2^{N(j)}t\}$.

\end{itemize}

}

\subsection{Application}

As a direct application of Theorem \ref{thm1}, we obtain the following the  quantitative weighted bounds  of Calder\'{o}n-type.
For a function $b\in L_{\mathrm{loc}}(\mathbb{R}^n)$,  let $A$ be a linear operator on some
measurable function space. Then the commutator between $A$ and
$b$ is defined by
$[b,A]f(x):=b(x)Af(x)-A(bf)(x)$.
\begin{theorem}\label{thm2}  Suppose $1<p<\infty$,  $w\in A_p$, $b\in Lip(\mathbb R^n)$ and $f\in C_0^1(\mathbb R^n)$.  Let $[b,T_\Omega]$ with $T_\Omega$ satisfying \eqref{ts}.
        Suppose that $\Omega$ have locally integrable first-order derivatives, $\Omega$ and its partial derivatives belong locally to $ L^\infty.$ Then there exists a positive constant  $C$ such that
$$
\big\|[b, T_\Omega] (\nabla f)\big\|_{L^p(w)}\lesssim
(w)_{A_p}\{w\}_{A_p}\|\nabla b\|_{L^\infty}\|f\|_{L^p(w)}.$$
Furthermore, if $[b,T_{\Omega}]f$ has first-order derivatives
in $L^p(w)$, then  there exists a positive constant  $C$ such that
$$
\big\|\nabla [b, T_\Omega]  f\big\|_{L^p(w)}\lesssim
(w)_{A_p}\{w\}_{A_p}\|\nabla b\|_{L^\infty}\|f\|_{L^p(w)}.$$
\end{theorem}

\begin{remark}
Our technique here can also be adapted to the study of the maximal commutator with rough kernels and the commutators with fractional differentiation, and to obtain quantitative weighted bounds. However, due to the different conditions in those settings and the length of the proofs, we will provide these results in the subsequent paper.
\end{remark}

\noindent{\bf Notation}. Throughout the whole paper, $p'=p/(p-1)$ represents the conjugate index of $p\in [1,\infty)$; $X\lesssim Y$ stands for $X\le C Y$ for a constant $C>0$ which is independent of the essential variables living on $X\ \&\ Y$; and $X\approx Y$ denotes  $X\lesssim Y\lesssim X$.

\section{Fundamental lemmas}

We first recall the definition and some properties of $A_p$ weight on $\mathbb R^n$. Let $w$ be a non-negative locally integrable function defined on $\mathbb R^n$. We say $w\in A_1$ if there is a constant $C>0$ such that $M(w)(x)\le Cw(x)$, where $M$ is the Hardy-Littlewood maximal function. Equivalently, $w\in A_1$ if and only if there is a constant $C>0$ such that for any cube $Q$
$$
\frac1{|Q|}\int_Qw(x)dx\le C\inf_{x\in Q}w(x).
$$
For $1<p<\infty$, we say that $w\in A_p$ if there exists a constant $C>0$ such that
\begin{align}\label{[Ap]}
[w]_{A_p}:=\sup_Q\bigg(\frac1{|Q|}\int_Qw(x)dx\bigg)\bigg(\frac1{|Q|}\int_Qw(x)^{1-p'}dx\bigg)^{p-1}\le C.
\end{align}
 We will adopt the following definition for the $A_\infty$ constant for a weight $w$ introduced by Fujii \cite{Fu}, and the later by Wilson \cite{Wi}:
 $$
[w]_{A_\infty}:=\sup_Q\frac{1}{w(Q)}\int_QM(w\chi_Q)(x)\,dx.$$
Here $w(Q):=\int_Qw(x)\,dx,$ and the supremum above is taken over all cubes with edges parallel to the coordinate axes. When the supremum is finite, we will say that $w$ belongs to the $A_\infty$ class.
 $A_\infty:=\bigcup_{p\ge1} A_p$. It is well known that if $w\in A_\infty$, then there exist $\delta\in(0,1]$ and $C>0$ such that for any interval $Q$ and measurable subset $E\subset Q$
$$
\frac{w(E)}{w(Q)}\le C\bigg(\frac{|E|}{|Q|}\bigg)^\delta.
$$

Denote by
\begin{align}\label{(Ap)}
(w)_{A_p}:=\max\Big\{[w]_{A_\infty},\,[w^{1-p'}]_{A_\infty}\Big\}.
\end{align}
and
\begin{align}\label{Ap brace}
\{w\}_{A_p}:=[w]_{A_p}^{1\over p}\max\Big\{[w]_{A_\infty}^{1\over p'},\,[w^{1-p'}]_{A_\infty}^{1\over p}\Big\}.
\end{align}

Recall from \cite{HRT} that
%using the facts that $[w^{1-p'}]_{A_p'}^{1/p'}=[w]_{A_p}^{1/p}$ and $[w]_{A_\infty}\le C[w]_{A_p}$(see \cite{HP}),
we have $$(w)_{A_p}\leq \tilde c_n \{w\}_{A_p} \leq \tilde c_n [w]_{A_p}^{\max\{1,1/(p-1)\}}.$$

Let us begin with presenting some auxiliary lemmas and their proofs, which will play a key role in proving Theorems \ref{thm1} and  \ref{thm2}.

We first recall some definitions.
A modulus of continuity is a function $\omega:[0,\infty)\rightarrow[0,\infty)$ with $\omega(0)$ that is subaddtive in the sense that
$$u\le t+s\Rightarrow \omega(u)\le \omega(t)+\omega(s).$$
Substituting $s=0$ one sees that $\omega(u)\le \omega(t)$ for all $0\le u\le t.$
Note that the composition and sum of two modulus of continuity is again a modulus of continuity.
In particular, if $\omega(t)$ is a modulus of continuity and $\theta\in (0,1),$ then $\omega(t)^\theta$ and $\omega(t^\theta)$ are also moduli of continuity.
 The Dini norm of a modulus of continuity are defined by setting
 \begin{align} \|\omega\|_{Dini}:=\int_0^1\omega(t)\frac{dt}{t}<\infty.
\end{align} For any $c>0$ the integral can be equivalently(up to a $c$-dependent multiplicative constant) replaced by the sum over $2^{-j/c}$ with $j\in \Bbb N.$
The basic example is $\omega(t)=t^\theta.$

 Let $T$ be a bounded linear operator on $L^2({\Bbb R}^n)$ represented as
  \begin{align} Tf(x)=\int_{{\Bbb R}^n}K(x,y)f(y)\,dy,\,\forall x \notin {\rm supp} f.
\end{align} We say that $T$ is an $\omega$-Calder\'{o}n-Zygmund   operator if the kernel  $K$ satisfies the following size and smoothness conditions:

For any $x,y\in {\Bbb R}^n\backslash \{0\}$,
\begin{align} \label{K0}|K(x,y)|\le \frac{C_K}{|x-y|^n},\,x\neq y,
\end{align}

For any $h\in {\Bbb R}^n$ with $2|h|\le |x-y|$,
\begin{align}\label{K1}  |K(y,x+h)-K(y,x)|+|K(x+h,y)-K(x,y)|\le \frac{\omega(|h|/|x-y|)}{|x-y|^n}.
\end{align}

\begin{lemma}[Theorem 1.3, \cite{HRT}]\label{lem 4} Let $T$  be an  $\omega$-Calder\'{o}n-Zygmund operator with $\omega$ satisfies the Dini condition. Then for $1<p<\infty,\,w\in A_p$, then
\begin{align*}\|Tf\|_{L^p(w)}\le C_{n,p}(\|T\|_{L^2\rightarrow L^2}+C_K+\|\omega\|_{Dini})\{w\}_{A_p}\|f\|_{L^p(w)}. \end{align*}
\end{lemma}

\begin{lemma}[\cite{CD1}]\label{CD0}\, For $b\in Lip(\mathbb{R}^n).$  Let $\psi\in {\mathscr
S}(\mathbb{R}^n)$ be a radial function such that ${\rm supp}\,\widehat{\psi}
\subset\{1/2\le |\xi|\le 2\}.$  Define the multiplier
operator $\Delta_j$ by $ \widehat{\Delta_jf}(
\xi)=\widehat{\psi}(2^{j}\xi)\widehat{f}(\xi)$ for $j\in \Bbb Z.$ Then for
$1<p<\infty$,
 we have  $$\bigg\|\bigg(\dsum_{j\in \Bbb Z}2^{-2j}|[b,\Delta_j]f|^2\bigg)^{1/2}\bigg\|_{L^p}\lesssim \|\nabla b\|_{L^\infty}\|f\|_{L^p}.$$ \end{lemma}

%\vskip0.2cm
\begin{lemma}\label{CD1}\, Let
	$(k,j)\in\mathbb Z\times\mathbb N$,
	$b\in Lip(\mathbb R^n),$ and  $m_{k,j} \in  C^\infty(\mathbb{R}^n)$. Suppose that $
\widehat{T_{k,j} f}(\xi)=m_{k,j}(\xi)\widehat{f}(\xi)$ and $
[b, T_{k,j}]f(x):=b(x)T_{k,j}f(x)-T_{k,j}(bf)(x).$
If for each $k,\,j$, the function $m_{k,j}(\xi)$ satisfies the following
condition:
$$
\begin{cases}
\|m_{k,j}\|_{L^\infty} \lesssim 2^{-k}2^{-\beta j},\quad {\rm where\ } \beta\ {\rm is\ a\ fixed\ positive\ constant\ independent\ of\ } j,k;\\[4pt]
\|\partial^\alpha m_{k,j}\|_{L^\infty} \lesssim 2^{k}\quad {\rm\ for\ any\ fixed\ multiindices}\ \alpha {\rm\ \ with\ } |\alpha|=2,
\end{cases}
$$
then
there exists a constant  $0<\lambda<1$ such that
$$
\|[b,T_{k,j}]f\|_{L^2} \lesssim   2
^{-\beta \lambda j}\|\nabla b\|_{L^\infty} \|f\|_{L^2}.
$$
\end{lemma}

\begin{proof}
Let  $\varphi\in C_0^\infty(\mathbb{R}^n)$ be a  radial
function with
$
\text{supp}(\varphi)\subset \{1/2\leq |x| \leq 2\}.$ Let
$\sum_{l \in \mathbb Z}\varphi(2^{-l}x)=1,$ for $ |x|>0.$
We set
$\varphi_0(x)=\sum_{l=-\infty}^0\varphi (2^{-l}x)$,
$\varphi_l(x) =
\varphi(2^{-l}x),$ for $l\in\mathbb N.$
Denote by $K_{k,j} (x)=m_{k,j}^{\vee}(x)$ - \text{the inverse Fourier transform of}\
$m_{k,j}.$
Then decompose  each $K_{k,j}$ as follows
$$\begin{array}{cl}
K_{k,j}(x)=K_{k,j}(x)\varphi_0(x)+\dsum_{l=1}^\infty
K_{k,j}(x)\varphi_l(x)=:\dsum_{l=0}^\infty K_{k,j}^l(x),
\end{array}$$
where
$$\begin{array}{cl}
\widehat{K_{k,j} ^l}(x)=\dint_{\mathbb{R}^n} m_{k,j}
(x-y)\widehat{\varphi_l}(y)\,dy.\end{array}$$
Since $\text{supp}(\varphi)\subset \{1/2\leq |x| \leq 2\}$, we see that for all multi-index $\vartheta$,
\begin{equation}\label{eta}
\int_{\mathbb{R}^n}\widehat{\varphi}(y)\ y^\vartheta\,dy=0.
\end{equation}
By using this cancellation condition \eqref{eta}, along with Taylor's expansion of $m_{k,j}(x-y)$ around $y$, we get that
\begin{align}\label{eta1}
\|\widehat{K_{k,j} ^l}\|_{L^\infty}&\le\dsum_{|\alpha|=2}\|\partial^\alpha m_{k,j}\|_{L^\infty}\dint_{\mathbb{R}^n} |y|^{2}|\widehat{\varphi_{l}}(y)|\,dy\\&
 = \nonumber\dsum_{|\alpha|=2}\|\partial^\alpha m_{k,j}\|_{L^\infty}
\dint_{\mathbb{R}^n} |{2^{-l}}y|^{2}|\widehat{\varphi}(y)|\,dy \\&
\lesssim \nonumber 2^{k}2^{-2l}\dint_{\mathbb{R}^n}
 |y|^{2}|\widehat{\varphi}(y)|\,dy \\&
 \lesssim \nonumber 2^{-2l}2^{k}.
\end{align}
On the other hand, by the Young inequality we have
\begin{align}\label{eta2}
\|\widehat{K_{k,j}
^l}\|_{L^\infty}&=\|m_{k,j}\ast\widehat{\varphi_l}\|_{L^\infty}\leq
\|m_{k,j}\|_{L^\infty}\|\widehat{\varphi_l}\|_{L^1}\lesssim
2^{-k} 2^{-\beta j}.
\end{align}
Therefore, interpolating between \eqref{eta1} and \eqref{eta2} gives that
\begin{align}\label{k1}
\|\widehat{K_{k,j} ^l}\|_{L^\infty} \lesssim2^{-2\theta
l}2^{k (2\theta-1) }
2^{-(1-\theta)\beta j}\qquad {\rm\ for\ any\ } 0<\theta<1.
\end{align}

 Denote by $T_{k,j}^lf(x)=K_{k,j}
^l\ast f(x).$ We now estimate $[b,T_{k,j}^l ]$ - the commutator of
the operator $T_{k,j}^l.$ Decompose $\mathbb{R}^n$ into a grid
of non-overlapping cubes with side length $2^l$ - i.e. -  $$\mathbb{R}^n =
\bigcup _{d=-\infty}^\infty Q_d$$ and set $$f_d := f\chi_{Q_d}.$$ Then
$$
f(x) = \dsum_{d=-\infty}^\infty f_d(x)\qquad \text{for a.e.}\ \ x \in \mathbb{R}^n.
$$
From the support condition of $T_{k,j}^l$,  it is obvious
that $$
\text{supp}\big([b,T_{k,j}^l]f_d\big) \subset 2nQ_d$$ and hence, the supports
of $\big\{[b,T_{k,j}^l ]f_d \big\}_{d=-\infty}^{+\infty}$ have bounded
overlaps. So we have the following almost orthogonality property
$$\begin{array}{cl}
\big\|[b,T_{k,j} ^l]f\big\|_{L^2}^2 \lesssim\dsum_{d=-\infty}^\infty
\big\|[b,T_{k,j} ^l] f_d\big\|_{L^2}^2.
\end{array}$$
Thus, we may assume that supp$(f)\subset Q$ for some cube $Q$ with $\ell(Q)=2^l.$ Upon choosing such a function $\phi \in C_0^\infty(\mathbb{R}^n)$ satisfying
$ 0 \leq
\phi\leq 1,$
$\text{supp}(\phi)
\subset 100nQ.$ Let
$\phi(x)=1,$ for $ x\in 30nQ$.
Denote by $\widetilde{Q}=200nQ,$ and
$\widetilde{b} =
(b(x)-b_{\widetilde{Q}})\phi(x),
$
we obtain that
$$\begin{array}{cl}
\|[b,T_{k,j}]f\|_{L^2}\le\dsum_{l\ge 0}\|[b,T_{k,j} ^l]f\|_{L^2}\leq \dsum_{l\ge 0}\big\|\widetilde{b}T_{k,j}^lf\big\|_{L^2}+\dsum_{l\ge 0}\big\|T_{k,j}^l(\widetilde{b}f)\big\|_{L^2}.
\end{array}$$
By using \eqref{k1} with  $1/2<\theta_1<1$  and with $0<\theta_2<1/2$, and applying the fact that
$\|\widetilde{b}\|_{L^\infty}\le 2^l\|\nabla b\|_{L^\infty} ,$
we obtain that there exists a constant $0<\lambda<1,$ such that
\begin{align*}
\dsum_{l\ge 0}\|\widetilde{b}T_{k,j}^lf\|_{L^2} &\leq
\dsum_{l\ge k}\|\widetilde{b}\|_{L^\infty}\|T_{k,j}^lf\|_{L^2}+\dsum_{l<k}\|\widetilde{b}\|_{L^\infty}\|T_{k,j}^lf\|_{L^2}\\
&\leq \dsum_{l\ge k}2^l\|\nabla b\|_{L^\infty} 2^{-2\theta_1l}2^{k(2 \theta_1-1) }2^{-j(1-\theta_1)\beta} \| f\|_{L^2}\\
&\quad+\dsum_{l< k} 2^l\|\nabla b\|_{L^\infty}  2^{-2\theta_2l}2^{k (2\theta_2-1) }2^{-j(1-\theta_2)\beta}\| f\|_{L^2}\\
&\leq (2^{-j(1-\theta_1)\beta}+2^{-j(1-\theta_2)\beta}) \|\nabla b\|_{L^\infty} \| f\|_{L^2}\\
&\simeq2^{- \lambda\beta j}\|\nabla b\|_{L^\infty} \| f\|_{L^2}.
\end{align*}

 Similarly, we can get
$$\begin{array}{cl}
\dsum_{l\ge 0}\|T_{k,j}(\widetilde{b}f)\|_{L^2}
\lesssim 2^{-\beta\lambda j}\|\nabla b\|_{L^\infty} \|f\|_{L^2}.
\end{array}$$
As a consequence, we obtain that
$$
\|[b,T_{k,j}]f\|_{L^2}\lesssim2^{-\beta\lambda j}\|\nabla b\|_{L^\infty} \|f\|_{L^2}.
$$
The proof of Lemma \ref{CD1} is complete.
\end{proof}

\begin{lemma} \label{sw1} Assume that $\varepsilon\in (0,1)$ and $1\leq p,\, q_0\leq\infty$, that $w$ is a positive weight, and that $T^\Omega$ is a sublinear operator associated with $\Omega$ satisfying
$$
\|T^\Omega f\|_{L^p}\leq k_0\|\Omega\|_{L^{q_0}({\Bbb S}^{n-1})}\|f\|_{L^p},
$$ and
$$
\|T^\Omega f\|_{L^p(w^{1+\varepsilon})}\leq k_1\|\Omega\|_{L^\infty({\Bbb S}^{n-1})}\| f\|_{L^p(w^{1+\varepsilon})}.
$$
 Then
$$
\|T^\Omega f\|_{L^p(w)}\leq k_0^{\frac{1}{1+\varepsilon}} k_1^{\frac{\varepsilon}{1+\varepsilon}}\|\Omega\|_{L^q({\Bbb S}^{n-1})}\| f\|_{L^p(w)},
$$
 where
$
\frac{1}{q}=\frac{\varepsilon}{1+\varepsilon}\frac{1}{q_0}.
$\end{lemma}

\emph{Proof.} Recall that $\frac{1}{q}=\frac{\varepsilon}{1+\varepsilon}\frac{1}{q_0}.$
Let $\beta(\xi)=2-(1+\varepsilon)\xi$, then $\beta(0)=2$, $\beta(1)=1-\varepsilon$, $\beta(\frac{1}{1+\varepsilon})=1$.
Let $ f_\xi(x)= f(x)w(x)^{\frac{\beta(\xi)-1}{p}}.$ It is easy to verify
\begin{align*}
\| f_0\|_{L^p}^p&=\int| f(x)|^pw(x)^{\frac{1}{p}\cdot p}dx=\|f\|_{L^p(w)}^p,
\end{align*}
and
\begin{align*}
\| f_1\|_{L^p(w^{1+\varepsilon})}^p&=\int| f(x)|^pw(x)^{\frac{-\varepsilon}{p}\cdot p}w(x)^{1+\varepsilon}dx=\int| f(x)|^pw(x)dx=\| f\|_{L^p(w)}^p.
\end{align*}

Let $\alpha(\xi)=\frac{1}{p'}+\frac{1+\varepsilon}{p}\xi$, then $\alpha(0)=\frac{1}{p'}, \alpha(1)=1+\frac{\varepsilon}{p}, \alpha(\frac{1}{1+\varepsilon})=1$.
Let $g_\xi(x)=g(x)w(x)^{\alpha(\xi)-1}$, then
\begin{align*}
\|g_0\|_{L^{p'}(w^{p'})}^{p'}&=\int|g(x)|^{p'}w(x)^{\frac{1}{p'}}w(x)^{p'}dx=\int|g(x)|^{p'}w(x)dx=\|g\|_{L^{p'}(w)}^{p'},
\end{align*}
and
\begin{align*}
\|g_1\|_{L^{p'}(w^{1-\frac{\varepsilon}{p-1}})}^{p'}&=\int|g(x)|^{p'}w(x)^{\frac{\varepsilon}{p}\cdot p'}w(x)^{1-\frac{\varepsilon}{p-1}}dx=\int|g(x)|^{p'}w(x)dx=\|g\|_{L^{p'}(w)}^{p'}.
\end{align*}

Let $\gamma(\xi)=\frac{q}{q_0}-\frac{q}{q_0}\xi$, then $\gamma(0)=\frac{q}{q_0}, \gamma(1)=0, \gamma(\frac{1}{1+\varepsilon})=1$.
Let $\Omega_\xi(x')=\Omega(x')^{\gamma(\xi)}\|\Omega\|_{L^q}^{1-\gamma(\xi)},$ then\begin{align*}
\|\Omega_0\|_{L^{q_0}({\Bbb S}^{n-1})}&=\|\Omega\|_{L^q({\Bbb S}^{n-1})},
\end{align*}
and
\begin{align*}
\|\Omega_1\|_{L^{\infty}({\Bbb S}^{n-1})}&=\|\Omega\|_{L^q({\Bbb S}^{n-1})}.
\end{align*}
Write
\begin{align*}
G(\xi)=\bigg|\int T^{\Omega_\xi} f_{\xi}(x)g_{\xi}(x)w(x)dx\bigg|.
\end{align*}
Since
\begin{align*}
G(0)&=\Big|\int T^{\Omega_0} f_{0}(x)g_{0}(x)w(x)dx\Big|\\
&\leq\|T_{\Omega_0}f_{0}\|_{L^p}\|g_0\|_{L^{p'}(w^{p'})}\\
&\leq k_0\|\Omega_0\|_{L^{q_0}}\|f_0\|_{L^p}\|g_0\|_{L^{p'}(w^{p'})}\\
&= k_0\|\Omega\|_{L^q({\Bbb S}^{n-1})}\|f\|_{L^p(w)}\|g\|_{L^{p'}(w)},
\end{align*}
and
\begin{align*}
G(1)&=\bigg|\int T^{\Omega_1}f_{1}(x)g_{1}(x)w(x)dx\bigg|\\
&\leq\| T_{\Omega_1}f_{1}\|_{L^p(w^{1+\varepsilon})}\|g_1\|_{L^{p'}(w^{1-\frac{\varepsilon}{p-1}})}\\
&\leq k_1\|\Omega_1\|_{L^\infty}\|f_1\|_{L^p(w^{1+\varepsilon})}\|g_1\|_{L^{p'}(w^{1-\frac{\varepsilon}{p-1}})}\\
&=k_1\|\Omega\|_{L^q({\Bbb S}^{n-1})}\|f\|_{L^p(w)}\|g\|_{L^{p'}(w)}.
\end{align*}
Note that
\begin{align*}
G\Big(\frac{1}{1+\varepsilon}\Big)=\bigg|\int T^{\Omega}f(x)g(x)w(x)dx\bigg|.
\end{align*}
By three-lines theorem, we get
\begin{align*}
G\Big(\frac{1}{1+\varepsilon}\Big)\leq k_0^{\frac{\varepsilon}{1+\varepsilon}}k_1^{\frac{1}{1+\varepsilon}}\|\Omega\|_{L^q({\Bbb S}^{n-1})}\|f\|_{L^p(w)}\|g\|_{L^{p'}(w)}.
\end{align*}
Then
\begin{align*}
\|T^\Omega f\|_{L^p(w)}\leq k_0^{\frac{\varepsilon}{1+\varepsilon}}k_1^{\frac{1}{1+\varepsilon}}\|\Omega\|_{L^q({\Bbb S}^{n-1})}\|f\|_{L^p(w)}.
\end{align*}

The proof of Lemma \ref{sw1} is complete. \qed

\vskip0.2cm

\section{Proof of Theorem 1.1}

\subsection{Main frame of the proof of Theorem 1.1}

Recall the definition of the operator $\mathcal{C}_{\Omega}$ given in the introduction. \begin{equation}\label{Cal}
\mathcal{C}_\Omega f
(x)=p.v.\int_{{\Bbb R}^n}\left(\cfrac{\Omega(x-y)}{|x-y|^{n+1}}\right)\big(b(x)-b(y)\big)f(y)dy\ \ \forall\ \ x\in\mathbb R^n.
 \end{equation}It can be
written as
\begin{align*}
\mathcal{C}_\Omega f=[b, T_{\Omega}^1]f=b(x)T_{\Omega}^1f(x)-T_{\Omega}^1(bf)(x),
\end{align*} where $$T_{\Omega}^1f(x)=p.v.\int_{{\Bbb R}^n}\frac{\Omega(x-y)}{|x-y|^{n+1}}f(y)\,dy.$$
Write\begin{align}\label{eq-01}
T_{\Omega}^1=\sum_{k \in \mathbb{Z}} T_{k} f=\sum_{k \in \mathbb{Z}} K_{k}\ast f, \quad K_{k}=\frac{\Omega(x^{\prime})}{|x|^{n+1}} \chi_{\{2^{k}<|x|\le2^{k+1}\}}.
\end{align}

We consider the following partition of unity.    Let $\varphi \in  {\mathscr S}({\Bbb R}^n)$   and  $\widehat{\varphi}\in {\mathscr S}({\Bbb R}^n)$ be a
radial function satisfying  $\widehat{\varphi}(\xi)=1$ for $|\xi|\le \frac12$ and $\widehat{\varphi}(\xi)=0$ for $|\xi|\ge 1$. Let us also define
$\psi$ by $\widehat{\psi}(\xi)^3=\widehat{\varphi}(\xi)-\widehat{\varphi}(2 \xi)\in {\mathscr
S}({\Bbb R}^n)$. Then, with this choice of $\widehat{\psi},$ it is supported by $\{\frac12\le|\xi|\le 2\}.$
We write $\varphi_{j}(x)=\frac{1}{2^{j n}} \varphi(\frac{x}{2^{j}}),$ and $\psi_{j}(x)=\frac{1}{2^{j n}} \psi(\frac{x}{2^{j}})$.
We now define the partial sum operators $S_{j}$ by $S_{j}(f)=f\ast\varphi_{j}$. Their differences
are given by
\begin{align}\label{eq-02}
S_{j}(f)-S_{j+1}(f)=f\ast\psi_{j}\ast\psi_{j}\ast\psi_{j}.
\end{align}
Since $S_{j}f \rightarrow f$ as $j \rightarrow -\infty$, for any sequence of integer numbers $\{N(j)\}_{j=0}^{\infty}$, with
$0 = N(0) < N(1) <\cdot\cdot\cdot < N(j)\rightarrow\infty $, we have the identity
\begin{equation}\label{eq-03}
T_{k}=T_{k} S_{k}+\sum_{j=1}^{\infty} T_{k}(S_{k-N(j)}-S_{k-N(j-1)}).
\end{equation}
We point out that such decomposition with respect to the  $\{N(j)\}_j$ is due to \cite{HRT}.
Next, by writing
\begin{align*}
 S_{k}%=\sum_{j= 1}^\infty(S_{k+j-1}-S_{k+j})
 =\sum_{j= 1}^\infty(S_{k+N(j-1)}-S_{k+N(j)}),
\end{align*}
we obtain that
\begin{align*}
 T_k%&=T_k\bigg(\sum_{j=1}^{\infty}(S_{k+N(j-1)}-S_{k+N(j)})+\sum_{j=1}^{\infty} (S_{k-N(j)}-S_{k-N(j-1)})\bigg)\\&
 =\sum_{j=1}^{\infty}T_k(S_{k+N(j-1)}-S_{k+N(j)})+\sum_{j=1}^{\infty}T_k(S_{k-N(j)}-S_{k-N(j-1)}).
\end{align*}
This gives
\begin{equation}\label{eq-03}
T_{\Omega}^1%=\sum_{j=1}^{\infty} {T}_{1,j}+\sum_{j=1}^{\infty}{T}_{2,j}
=\sum_{j=1}^{\infty} T_{1,j}^{N} + \sum_{j=1}^{\infty} T_{2,j}^{N},
\end{equation} where
 for $j\geq1$,
\begin{align}\label{t1} %T_{1,j} &:=\sum_{k \in \mathbb{Z}} T_{k}(S_{k-j}-S_{k-(j-1)}),\nonumber\\
T_{1,j}^{N}
&:=\sum_{k \in \mathbb{Z}} T_k(S_{k-N(j)}-S_{k-N(j-1)})%=\sum_{i=N(j-1)+1}^{N(j)} T_{1,i},
\end{align}
and
\begin{align}\label{t2}
%T_{2,j} &:=\sum_{k \in \mathbb{Z}} T_{k}(S_{k+j-1}-S_{k+j}),\nonumber\\
T_{2,j}^{N}
&:=\sum_{k \in \mathbb{Z}} T_{k}(S_{k+N(j-1)}-S_{k+N(j)}).%=\sum_{i=N(j-1)+1}^{N(j)}T_{2,i}.
\end{align}
Thus
\begin{align*}
\mathcal{C}_\Omega f=[b, T_\Omega^1] f =\sum_{j=1}^{\infty} [b, T_{1,j}^{N}+T_{2,j}^{N}] f =\sum_{j=1}^{\infty} [b, T_{1,j}^{N}] f+\sum_{j=1}^{\infty} [b, T_{2,j}^{N}] f .
\end{align*}
Therefore,  for $1<p<\infty,$ and $w\in A_p,$ \begin{align*}\|\mathcal{C}_\Omega f\|_{L^p(w)}\leq\sum_{j=1}^{\infty} \| [b, T_{1,\,j}^{N}]f\|_{L^p(w)}+\sum_{j=1}^{\infty} \| [b, T_{2,\,j}^{N}]f\|_{L^p(w)}. \end{align*}  To prove Theorem \ref{thm1}, we need to establish the quantitative weighted bounds for $[b, T_{1,j}^{N}]$ and $[b, T_{2,j}^{N}]$, respectively.

\textbf{Part 1. }First, for $[b, T_{1,j}^{N}]$, we do not assume the cancellation condition \eqref{mean zero} on $\Omega,$  we claim the following inequalities holds $1<p<\infty$ and $w\in A_p$:
\begin{align}\label{vv}
 \|[b, T_{1,\,j}^{N}]f\|_{L^p}&\lesssim \|\nabla b\|_{L^\infty}\|\Omega\|_{L^{q_0}}2^{-\theta N(j-1)}\|f\|_{L^p} \end{align} for any fixed $q_0>1,$
 and
\begin{align}\label{vq3}
 \|[b, T_{1,\,j}^{N}]f\|_{L^p(w)}&\lesssim \|\nabla b\|_{L^\infty}\|\Omega\|_{L^{\infty}}(1+N(j))\{w\}_{A_p}\|f\|_{L^p(w)}.
 \end{align}

We assume \eqref{vv} and \eqref{vq3} for the moment, which will be given in Section 3.2. Then, by choosing $\varepsilon:=\frac{1}{2}c_n/(w)_{A_p}$ (see \cite{HRT}), we see that   the estimate \eqref{vq3} gives
\begin{align}\label{vw}
\|[b, T_{1,\,j}^{N}]\|_{L^p(w^{1+\varepsilon})}&\lesssim\|\nabla b\|_{L^\infty}\|\Omega\|_{L^\infty}(1+N(j)){\{w^{1+\varepsilon}\}_{A_p}}\|f\|_{L^p(w^{1+\varepsilon})}\\
&\lesssim\|\nabla b\|_{L^\infty}\|\Omega\|_{L^\infty}(1+N(j)){\{w\}}_{A_p}^{1+\varepsilon}\|f\|_{L^p(w^{1+\varepsilon})}\nonumber.
\end{align}
Now we are in position to apply the interpolation theorem with change of measures and kernels.
We now apply Lemma  \ref{sw1} to $T^\Omega=[b, T_{1,\,j}^{N}]$. Then we obtain that  there exist some $\theta,\,\gamma>0 $ such that for $q>\frac{(1+\varepsilon)}{\varepsilon},$
\begin{align*}
\|[b, T_{1,\,j}^{N}]\|_{L^p(w)\rightarrow L^p(w)}
&\lesssim\|\nabla b\|_{L^\infty}\|\Omega\|_{L^q}(1+N(j))^{\frac{1}{1+\varepsilon}}2^{-\theta N(j-1)\varepsilon/(1+\varepsilon)}{\{w\}}_{A_p}\\
&\lesssim\|\nabla b\|_{L^\infty}\|\Omega\|_{L^q}(1+N(j))^{\frac{1}{1+\varepsilon}}2^{-\gamma N(j-1)/{(w)}_{A_p}}{\{w\}}_{A_p}.
\end{align*}
This gives
\begin{align*}
\sum^\infty_{j=1}\|[b, T_{1,\,j}^{N}]\|_{L^p(w)\rightarrow L^p(w)}
&\lesssim \|\nabla b\|_{L^\infty}\|\Omega\|_{L^q}{\{w\}}_{A_p}\sum^\infty_{j=1}(1+N(j))^{\frac{1}{1+\varepsilon}}2^{-\gamma N(j-1)/{(w)}_{A_p}}.
\end{align*}
Thus, it suffices to choose a suitable increasing sequence $\{N(j)\}$ to get the quantitative estimate. We choose $N(j)=2^j$ for $j\geq1$. Then, using $e^{-x}\leq2x^{-2}$, we have
\begin{align*}
\sum^\infty_{j=1}&(1+N(j))2^{-\gamma N(j-1)/{(w)}_{A_p}}
\lesssim\sum_{j:2^{j}\leq(w)_{A_p}}2^{j}+\sum_{j:2^j\geq(w)_{A_p}}2^j\big(\frac{(w)_{A_p}}{2^j}\big)^2
\lesssim(w)_{A_p},
\end{align*}
by summing two geometric series in the last step.
As a consequence, we have for $q>\frac{(1+\varepsilon)}{\varepsilon},$
$$
\sum^\infty_{j=1}\|[b, T_{1,\,j}^{N}]\|_{L^p(w)\rightarrow L^p(w)}\lesssim\|\nabla b\|_{L^\infty}\|\Omega\|_{L^q}{\{w\}}_{A_p}(w)_{A_p}\|f\|_{L^p(w)}.
$$

{\textbf{Part 2}. }Next, for $[b, T_{2,j}^{N}]$, if $\Omega$  satisfies the cancellation condition \eqref{mean zero},   we claim the following inequalities holds : for $1<p<\infty$ and $w\in A_p,$
\begin{align}\label{vv1}
 \|[b, T_{2,\,j}^{N}]f\|_{L^p}&\lesssim \|\nabla b\|_{L^\infty}\|\Omega\|_{L^{1}}2^{-\theta N(j-1)}\|f\|_{L^p}, \end{align}
 and 
\begin{align}\label{vq31}
 \|[b, T_{2,\,j}^{N}]f\|_{L^p(w)}&\lesssim \|\nabla b\|_{L^\infty}\|\Omega\|_{L^{1}}(1+N(j))\{w\}_{A_p}\|f\|_{L^p(w)}.
 \end{align}

We assume \eqref{vv1} and \eqref{vq31} for the moment, which will be given in Section 3.3. Then, by choosing $\varepsilon:=\frac{1}{2}c_n/(w)_{A_p}$ (see \cite{HRT}), we see that   the estimate \eqref{vq3} gives
\begin{align}\label{vw1}
\|[b, T_{2,\,j}^{N}]\|_{L^p(w^{1+\varepsilon})}&\lesssim\|\nabla b\|_{L^\infty}\|\Omega\|_{L^1}(1+N(j)){\{w^{1+\varepsilon}\}_{A_p}}\|f\|_{L^p(w^{1+\varepsilon})}\\
&\lesssim\|\nabla b\|_{L^\infty}\|\Omega\|_{L^1}(1+N(j)){\{w\}}_{A_p}^{1+\varepsilon}\|f\|_{L^p(w^{1+\varepsilon})}\nonumber.
\end{align}

\begin{lemma}[Stein--Weiss] \label{sw} Assume that $1\leq p_0,p_1\leq\infty$, that $w_0$ and $w_1$ are positive weights, and that $T$ is a sublinear operator  satisfying
$$
T:L^{p_i}(w_i)\rightarrow L^{p_i}(w_i),~~~~i=0,1,
$$
with quasi-norms $M_0$ and $M_1$, respectively. Then
$$
T:L^{p}(w)\rightarrow L^{p}(w),
$$
with quasi-norm $M\leq M_0^\lambda M_1^{1-\lambda}$, where
$$
\frac{1}{p}=\frac{\lambda}{p_0}+\frac{1-\lambda}{p_1},\quad w=w_0^{p\lambda/p_0}w_1^{p(1-\lambda)/p_1}.
$$\end{lemma}
With \eqref{vv1} and \eqref{vw1}, we now apply Lemma \ref{sw} to $T=[b, T_{2,\,j}^{N}]$ with $p_0=p_1=p,~~w_0=w^0=1$, $w_1=w^{1+\varepsilon}$ and $\lambda=\varepsilon/(1+\varepsilon)$. Then we obtain that  there exist some $\theta,\,\gamma>0 $ such that
\begin{align*}
\|[b, T_{2,\,j}^{N}]\|_{L^p(w)\rightarrow L^p(w)}&\lesssim
\|[b,T_{2,\,j}^{N}]\|_{L^p\rightarrow L^p}^{\varepsilon/(1+\varepsilon)}\|[b, T_{2,\,j}^{N}]\|_{L^p(w^{1+\varepsilon})\rightarrow L^p(w^{1+\varepsilon})}^{1/(1+\varepsilon)}\\
&\lesssim\|\nabla b\|_{L^\infty}\|\Omega\|_{L^1}(1+N(j))2^{-\theta N(j-1)\varepsilon/(1+\varepsilon)}{\{w\}}_{A_p}\\
&\lesssim\|\nabla b\|_{L^\infty}\|\Omega\|_{L^1}(1+N(j))2^{-\gamma N(j-1)/{(w)}_{A_p}}{\{w\}}_{A_p}.
\end{align*}
This gives
\begin{align*}
\sum^\infty_{j=1}\|[b, T_{2,\,j}^{N}]\|_{L^p(w)\rightarrow L^p(w)}
&\lesssim \|\nabla b\|_{L^\infty}\|\Omega\|_{L^1}{\{w\}}_{A_p}\sum^\infty_{j=1}(1+N(j))2^{-\gamma N(j-1)/{(w)}_{A_p}}.
\end{align*}
Thus, it suffices to choose a suitable increasing sequence $\{N(j)\}$ to get the quantitative estimate. We choose $N(j)=2^j$ for $j\geq1$. Similarly, we have
\begin{align*}
\sum^\infty_{j=1}&(1+N(j))2^{-\gamma N(j-1)/{(w)}_{A_p}}
\lesssim(w)_{A_p}.
\end{align*}
As a consequence, we have
$$
\sum^\infty_{j=1}\|[b, T_{2,\,j}^{N}]f\|_{L^p(w)}\lesssim\|\nabla b\|_{L^\infty}\|\Omega\|_{L^1}{\{w\}}_{A_p}(w)_{A_p}\|f\|_{L^p(w)}.
$$

\bigskip

\subsection {  Proof of \eqref{vv} and \eqref{vq3}}

 To begin with,  we need the following lemma which gives \eqref{vv}.

\begin{lemma}\label{v2} Let  $b\in Lip({\Bbb R}^n)$. For any fixed $q>1,$ let $\Omega\in L^q({\Bbb S}^{n-1}).$ The following inequality holds for $1<p<\infty$ and some $\delta\in (0,1),$
\begin{align}\label{frourier 1}
\|[b, T_{1,\,j}^{N}] f\|_{L^{p}}
& \lesssim \|\Omega\|_{L^{q}}\|\nabla b\|_{L^\infty} 2^{-\delta N(j-1)}\|f\|_{L^{p}}.
\end{align}

\end{lemma}

 \emph{Proof.}
Recall that
\begin{align*}[b, T_{1,\,j}^{N}]f&=\sum_{k\in \Bbb Z}[b, T_k(S_{k-N(j)}-S_{k-N(j-1)})] f, \end{align*}
where $$S_{k-N(j)}-S_{k-N(j-1)}=\sum_{i=N(j-1)+1}^{N(j)}(S_{k-i}-S_{k-i+1})=\sum_{i=N(j-1)+1}^{N(j)}\Delta_{k-i}^3.$$
Then for $1<p<\infty$
\begin{align*}\|[b, T_{1,\,j}^{N}] f\|_{L^{p}}&\le \sum_{i=N(j-1)+1}^{N(j)}\Big\|\sum_{k\in \Bbb Z}[b, T_k\Delta_{k-i}^3] f\Big\|_{L^{p}}. \end{align*}
 First, we give the dedicated $L^2$-norm of $\sum_{k\in \Bbb Z}[b, T_k\Delta_{k-i}^3] f.$ For any $i\in \Bbb N,\, k\in \mathbb{Z}$, we define  $T_k^if(x):=T_{k}\Delta_{k-i}f(x).$
 Then we  write
\begin{align*}
[b,T_{k}\Delta_{k-i}^3]f
&=[b,\Delta_{k-i}](T_{k}^i\Delta_{k-i}f)+\Delta_{k-i}([b,T_{k}^i]\Delta_{k-i}f)
+\Delta_{k-i}T_{k}^i([b,\Delta_{k-i}]f).
\end{align*}
Then we get
\begin{align*}
\Big\|\sum_{k\in \Bbb Z}[b, T_k\Delta_{k-i}^3] f\Big\|_{L^{2}}
& =\Big\|\sum_{k\in \Bbb Z}[b, \Delta_{k-i}]T_k^j\Delta_{k-i} f\Big\|_{L^{2}} +\Big\|\sum_{k\in \Bbb Z} \Delta_{k-i}[b,T_k^i]\Delta_{k-i}f\Big\|_{L^{2}}+\Big\|\sum_{k\in \Bbb Z} \Delta_{k-i}T_k^i[b,\Delta_{k-i}]f\Big\|_{L^{2}}\\&=:I_1+I_2+I_3.
\end{align*}
Denote
\begin{align*}
\widehat{{T}_{k}^i f}(\xi)= \widehat{K}_{k}(\xi) \widehat{\psi}(2^{k-i} \xi) \widehat{f}(\xi)=: m_{i,k}(\xi) \widehat{f}(\xi).
\end{align*}

We now provide pointwise estimates for $m_{i,k}(\xi)$ and its derivative.
First, it is easy to verify that there exists $0<\beta<1$ such that for any fixed $q>1,$ $$|\widehat{K}_{k}(\xi)| \lesssim\|\Omega\|_{L^{q}} 2^{-k} |2^k\xi|^{-\beta}.$$ Second,   for any multi-index $\eta$ with $|\eta|\leq
2$, we get
$$
\begin{array}{cl}
\big|\partial^{\eta}\widehat{K}_{k}(\xi)\big|
 &\lesssim\left|\dint^{2^{k+1}}_{2^k}\dint_{S^{n-1}}\Omega(y'){y'}^{\eta}e^{-2\pi {\bf i} \xi\cdot
 ry'}d\sigma(y')r^{|\eta|-2}dr\right|\lesssim\|\Omega\|_{L^{1}}2^{k(|\eta|-1)}.
\end{array}
$$Since $\widehat{\psi} \in \mathcal{S}(\mathbb{R}^{d})$ %, $\widehat{\psi}(0)=0,$ and
with ${\rm supp}\widehat{\psi}\subset \{1/2\le |\xi|\leq 2\}$, we see that $|2^k\xi|\approx 2^{i}$.
 Hence,
\begin{align}\label{m1}
|m_{i,k}(\xi)| \leq |\widehat{K}_{k}(\xi)|\,| \widehat{\psi}(2^{k-i} \xi)|\lesssim2^{-k}2^{-\beta i}\|\Omega\|_{L^{q}}.
\end{align}
On the other hand, for any multi-index $\alpha$, we have
\begin{align*}
\partial ^{\alpha}m_{i,k}(\xi)&=
\partial
^{\alpha}\big(\widehat{K}_{k}(\xi)\psi(2^{k-i}\xi)\big)=\dsum\limits_{\eta}C^{\alpha_1}_{\eta_1}...C^{\alpha_n}_{\eta_n}
\big(\partial ^{\eta}\widehat{K}_{k}(\xi)\big) \big(\partial
^{\alpha-\eta}\psi(2^{k-i}\xi)\big),
\end{align*}
where the sum is taken over all multiindices $\eta$ with $0\leq
\eta_\ell \leq \alpha_\ell$ for $1\leq \ell \leq n$.  By taking $\alpha$ with
$|\alpha|=2$, we obtain that
\begin{align}\label{m2}
\big|\partial^{\alpha}m_{i,k}(\xi)\big|&\lesssim2^{(k-i)(|\alpha|-|\eta|)}\dsum\limits_{0\leq |\eta| \leq
|\alpha|}\big| \partial
^{\eta}\widehat{K}_{k}(\xi)\big|\\
&\lesssim 2^{(k-i)(|\alpha|-|\eta|)}2^{(-1+|\eta|)k}\|\Omega\|_{L^{1}} \nonumber\\
&\simeq 2^{-i(|\alpha|-|\eta|)}2^{(-1+|\alpha|)k}\|\Omega\|_{L^{1}}\nonumber\\
&\lesssim 2^k\|\Omega\|_{L^{1}}.\nonumber
\end{align}
As a consequence, after combing \eqref{m1} and \eqref{m2} above, we can
apply Lemma \ref{CD1} to $[b,{T}_{k}^i]$ to obtain that there exists  some constant
$0<\lambda<1$ such that
\begin{align}\label{t1 bTjk}
\|[b,{T}_{k}^i]f\|_{L^2}\lesssim 2^{-\beta \lambda
i}\|\Omega\|_{L^q}\|\nabla b\|_{L^\infty}  \|f\|_{L^2},\ \ \ i\ge 1.
\end{align}
By  using the Plancherel theorem and \eqref{m1}, we also for any fixed $q>1$
\begin{align}\label{t2 Tjk}\|{T}_{k}^i f\|_{L^2}\lesssim 2^{-k}  2^{-\beta
i}\|\Omega\|_{L^q}\|f\|_{L^ 2}.
\end{align}
We now estimate $I_1$, $I_2$ and $I_2$, respectively.
By using  \eqref{t1 bTjk} and the Littlewood--Paley theory,  we get
\begin{align*}
I_1&\lesssim \bigg(\dsum_{k\in \mathbb Z}\|[b,T_k^i](\Delta_{k-i}f)\|_{L^2}^2\bigg)^{1/2}\\
&\lesssim 2^{-\beta \lambda i}\|\Omega\|_{L^q} \|\nabla b\|_{L^\infty}\bigg(\dsum_{k\in \mathbb Z}\|\Delta_{k-i}f\|_{L^2}^2\bigg)^{1/2}\\
&\lesssim 2^{-\beta \lambda i}\|\Omega\|_{L^q}\|\nabla b\|_{L^\infty}\|f\|_{L^2}.
\end{align*}
 Now, we estimate $I_2.$  By using \eqref{t2 Tjk} and Lemma  \ref{CD0}, we get\begin{align*}I_2
&\lesssim \bigg(\dsum_{i\in \mathbb Z}\|T_k^i([b,\Delta_{k-i}]f)\|_{L^2}^2\bigg)^{1/2}\\
&\lesssim 2^{-(1+\beta)i} \|\Omega\|_{L^q}\bigg(\dsum_{i\in \mathbb
Z}2^{-2(k-i)}\|[b,\Delta_{k-i}]f\|_{L^2}^2\bigg)^{1/2}\\&\lesssim2^{-(1+\beta)i}\|\nabla b\|_{L^\infty}\|\Omega\|_{L^q}\|f\|_{L^2}.\end{align*}
Finally,
 by duality and by the estimate of $I_1$, we obtain that
\begin{align*}
I_3
&\lesssim 2^{-\beta\lambda i}\|\Omega\|_{L^q}\|\nabla b\|_{L^\infty}\|f\|_{L^2}.
\end{align*}
Combining the estimates of $I_1,\,I_2$ and $I_3$ above,  we obtain that there exists some constant $0<\alpha<1$ such that
\begin{align}\label{a1}\Big\|\sum_{k\in \Bbb Z}[b, T_k\Delta_{k-i}^3] f\Big\|_{L^{2}}\lesssim2^{-\alpha i}\|\nabla b\|_{L^\infty}\|\Omega\|_{L^q}\|f\|_{L^2}.\end{align}
Next, we give the $L^p$-norm of $\sum_{k\in \Bbb Z}[b, T_k\Delta_{k-i}^3] f.$
We  write
\begin{align*}
[b,T_{k}\Delta_{k-i}^3]f
&=[b,\Delta_{k-i}^2](T_{k}\Delta_{k-i}f)+\Delta_{k-i}^2([b,T_{k}]\Delta_{k-i}f)
+\Delta_{k-i}^2T_{k}([b,\Delta_{k-i}]f).
\end{align*}
Then we get
\begin{align*}
\Big\|\sum_{k\in \Bbb Z}[b, T_k\Delta_{k-i}^3] f\Big\|_{L^{p}}
& =\Big\|\sum_{k\in \Bbb Z}[b, \Delta_{k-i}^2]T_k\Delta_{k-i} f\Big\|_{L^{p}} +\Big\|\sum_{k\in \Bbb Z} \Delta_{k-i}^2[b,T_k]\Delta_{k-i}f\Big\|_{L^{p}}+\Big\|\sum_{k\in \Bbb Z} \Delta_{k-i}^2T_k[b,\Delta_{k-i}]f\Big\|_{L^{p}}\\&=:II_1+II_2+II_3.
\end{align*}
We now estimate $II_1$, $II_2$ and $II_3$, respectively. For $II_1,$ since  
\begin{align*}
\Big\|\sum_{k\in \Bbb Z}[b, \Delta_{k-i}^2]f_k\Big\|_{L^{p}}\le \Big\|\sum_{k\in \Bbb Z}[b, \Delta_{k-i}]\Delta_{k-i}f_k\Big\|_{L^{p}}+\Big\|\sum_{k\in \Bbb Z}\Delta_{k-i}[b, \Delta_{k-i}]f_k\Big\|_{L^{p}},
\end{align*}
then by duality, Lemma  \ref{CD0} and Littlewood-Paley theory, we get for $1<p<\infty,$ \begin{align}\label{io}
\Big\|\sum_{k\in \Bbb Z}[b, \Delta_{k-i}^2]f_k\Big\|_{L^{p}}\le \|\nabla b\|_{L^\infty}\Big\|\big(\sum_{k\in \Bbb Z}|2^{k-i}f_k|^2\big)^{1/2}\Big\|_{L^{p}}.
\end{align} On the other hand,
$|T_kf(x)|\le 2^{-k}M_\Omega f(x)$,  where
$$M_{\Omega}g(x)=\sup_{r>0}\frac{1}{r^n}\dint_{|x-y|<r}|\Omega(x-y)||g(y)|\,dy.$$ Then we get \begin{align}\label{0}\bigg\|\bigg(\sum_{k\in \mathbb Z}|T_kf_k|^2\bigg)^{1/2}\bigg\|_{L^p}\lesssim \|\Omega\|_{L^1}\bigg\|\bigg(\sum_{k\in \mathbb Z}|2^{-k}f_k|^2\bigg)^{1/2}\bigg\|_{L^p}.\end{align}
Thus by using \eqref{io}, \eqref{0}  and  Littlewood-Paley theory,  we get
\begin{align*}
II_1&\lesssim \|\nabla b\|_{L^\infty}\bigg\|\bigg(\dsum_{k\in \mathbb Z}|2^{k-i}T_k(\Delta_{k-i}f)|^2\bigg)^{1/2}\bigg\|_{L^p}\\
&\lesssim 2^{-i}\|\Omega\|_{L^1} \|\nabla b\|_{L^\infty}\bigg\|\bigg(\dsum_{k\in \mathbb Z}|\Delta_{k-i}f|^2\bigg)^{1/2}\bigg\|_{L^p}\\
&\lesssim \|\Omega\|_{L^1}\|\nabla b\|_{L^\infty}\|f\|_{L^p}.
\end{align*}
 Now, we estimate $II_2.$ It is well known that for any $g\in L^p(\mathbb R^n)$
$$\begin{array}{cl}|[b,T_{k}]g(x)|
&\lesssim\|\nabla b\|_{L^\infty}M_{\Omega}g(x).\end{array}$$
From this, we  get for $1<p<\infty,$$$\bigg\|\bigg(\sum_{k\in \mathbb Z}|[b,T_k]g_k|^2\bigg)^{1/2}\bigg\|_{L^p}\lesssim \|\Omega\|_{L^1}\|\nabla b\|_{L^\infty}\bigg\|\bigg(\sum_{k\in \mathbb Z}|g_k|^2\bigg)^{1/2}\bigg\|_{L^p}.$$ Then by using Littlewood-Paley theory, we get\begin{align*}II_2
&\lesssim \bigg\|\bigg(\dsum_{k\in \mathbb Z}|[b,T_k](\Delta_{k-i}f)|^2\bigg)^{1/2}\bigg\|_{L^p}\\
&\lesssim \|\nabla b\|_{L^\infty} \|\Omega\|_{L^1}\bigg\|\bigg(\dsum_{k\in \mathbb Z}|\Delta_{k-i}f|^2\bigg)^{1/2}\bigg\|_{L^p}\\&\lesssim\|\nabla b\|_{L^\infty} \|\Omega\|_{L^1}\|f\|_{L^p}.\end{align*}
Finally, by Littlewood-Paley theory, \eqref{0} and Lemma  \ref{CD0}, 
 we obtain that
\begin{align*}
II_3&\lesssim \bigg\|\bigg(\dsum_{k\in \mathbb Z}|T_k([b,\Delta_{k-i}]f)|^2\bigg)^{1/2}\bigg\|_{L^p}\\
&\lesssim \|\Omega\|_{L^1} \bigg\|\bigg(\dsum_{k\in \mathbb Z}|2^{-k}[b,\Delta_{k-i}]f|^2\bigg)^{1/2}\bigg\|_{L^p}\\
&\lesssim 2^{-i}\|\Omega\|_{L^1} \bigg\|\bigg(\dsum_{k\in \mathbb Z}|2^{i-k}[b,\Delta_{k-i}]f|^2\bigg)^{1/2}\bigg\|_{L^p}\\
&\lesssim 2^{-i}\|\Omega\|_{L^1}\|\nabla b\|_{L^\infty} \|f\|_{L^p}.
\end{align*}
Combining the estimates of $II_1,\,II_2$ and $II_3$ above,  we obtain that 
\begin{align} \label{o}\Big\|\sum_{k\in \Bbb Z}[b, T_k\Delta_{k-i}^3] f\Big\|_{L^{p}}\lesssim\|\nabla b\|_{L^\infty}\|\Omega\|_{L^1}\|f\|_{L^p}\end{align}
Interpolating between\eqref{a1} and \eqref{o}, we get for some $\gamma\in (0,1),$ \begin{align} \label{o1}\Big\|\sum_{k\in \Bbb Z}[b, T_k\Delta_{k-i}^3] f\Big\|_{L^{p}}\lesssim 2^{-\gamma i}\|\nabla b\|_{L^\infty}\|\Omega\|_{L^1}\|f\|_{L^p}\end{align}
Then we get for $\delta\in (0,1)$ and for any fixed $q>1,$
\begin{align*}
\|[b, T_{1,\,j}^{N}]f\|_{L^{p}}&\le\sum_{i=N(j-1)+1}^{N(j)}\Big\|\sum_{k\in \Bbb Z}[b, T_k\Delta_{k-i}^3] f\Big\|_{L^{p}}\\
&\lesssim\sum_{i=N(j-1)+1}^{N(j)}2^{-\gamma i}\|\nabla b\|_{L^\infty}\|\Omega\|_{L^q}\|f\|_{L^2}\\&\lesssim2^{-\delta N(j-1)}\|\nabla b\|_{L^\infty}\|\Omega\|_{L^q}\|f\|_{L^p}.
\end{align*}

The proof of Lemma \ref{v2} is complete. \qed

\quad

\begin{lemma}\label{lemtj} Let  $b\in Lip({\Bbb R}^n)$. If $\Omega\in L^\infty({\Bbb S}^{n-1}),$ then the operator $[b, T_{1,\,j}^{N}]$  is a  Calder\'{o}n-Zygmund operator with kernel $K_1(x,y)$ satisfies
\begin{align*}
|K_1(x,y)|\lesssim \frac{\|\Omega\|_{L^\infty}\|\nabla b\|_{L^\infty}}{|x-y|^n}
 \end{align*}
 and for $2|h|\le |x-y|,$
 \begin{align*}
|K_1(x,y+h)-K_1(x,y)|+|K_1(x,y)-K_1(x+h,y)|\lesssim\frac{\omega_{1,j}(\frac{|h|}{|x-y|})}{|x-y|^{n}},
 \end{align*}
 where $\omega_{1,j}(t)=\|\Omega\|_{L^\infty}\|\nabla b\|_{L^\infty}\min\{1,2^{N(j)}t\}.$
 \end{lemma}

\quad

In fact, applying Lemma \ref{lemtj} and Lemma \ref{lem 4} to $T=[b, T_{1,\,j}^{N}]$, we get that for $1<p<\infty$ and $w\in A_p$\begin{align*}
\|[b, T_{1,\,j}^{N}]f\|_{L^p(w)}&\lesssim \|\Omega\|_{L^{\infty}}\|\nabla b\|_{L^\infty}(1+N(j)){\{w\}}_{A_p}\|f\|_{L^p(w)},
\end{align*} which gives the proof of  \eqref{vq3}.

\quad\quad

\emph{\textbf{Proof of Lemma \ref{lemtj}}.}
In order to get the required size and smoothness estimates for
 $K_1(x,y)$,
we need to study the kernels of  $T_{k} S_{k-N(j)}$.
We begin with verifying the size and the smoothness estimates for the kernel of $T_{k} S_{k-N(j)}$, that is,
to get the pointwise estimate of $|K_{k}\ast \varphi_{k-N(j)}(x)|$ and $|\nabla K_{k}\ast \varphi_{k-N(j)}(x)|.$

  Let $x \in \mathbb{R}^{n}$. From the definition of $K_k $ and $\varphi_{k-N(j)}(x)$, we get
\begin{align*}|K_{k}\ast \varphi_{k-N(j)}(x)| &=\Bigg|\int_{\mathbb{R}^{n}} \frac{\Omega(y^{\prime})}{|y|^{n+1}} 1_{2^{k}<|y|<2^{k+1}}(y) 2^{-(k-N(j)) n} \varphi\Big(\frac{x-y}{2^{k-N(j)}}\Big) d y\Bigg|.  \end{align*}

We  now consider the following two cases: $|x|>2^{k+2}$ and $|x|\le 2^{k+2}.$

Case 1:   $|x|>2^{k+2}$.

Note that in this case, we get $|x-y|\ge |x|-|y|\ge |x|-\frac{|x|}{2}\ge \frac{|x|}{2}$.
Hence,
\begin{align*}
|K_{k}\ast \varphi_{k-N(j)}(x)|
%&=\bigg|\int_{\mathbb{R}^{n}} \frac{\Omega(y^{\prime})}{|y|^{n+1}} 1_{2^{k}<|y|<2^{k+1}}(y) \varphi_{k-N(j)}(x-y) \,d y\bigg|\\
&\leq\int_{\mathbb{R}^{n}} \frac{|\Omega(y^{\prime})|}{|y|^{n+1}} 1_{2^{k}<|y|<2^{k+1}}(y) \frac{2^{2(k-N(j))}}{(2^{k-N(j)}+|x-y|)^{n+2}} d y\\
&\lesssim \frac{2^{2(k-N(j))}}{|x|^{n+2}}\int_{\mathbb{R}^{n}} \frac{|\Omega(y^{\prime})|}{|y|^{n+1}} 1_{2^{k}<|y|<2^{k+1}}(y) d y \\
&\lesssim \|\Omega\|_{L^1}\frac{2^{-2N(j)} 2^{k}}{|x|^{n+2}}.%\chi_{|x|>2^{k+2}},
\end{align*}

Case 2: $|x|\le2^{k+2}.$

We get
\begin{align*}
|K_{k}\ast \varphi_{k-N(j)}(x)| &\le\int_{\mathbb{R}^{n}} \frac{|\Omega(y^{\prime})|}{|y|^{n+1}} 1_{2^{k}<|y|<2^{k+1}}(y) 2^{-(k-N(j)) n} \varphi\Big(\frac{x-y}{2^{k-N(j)}}\Big)\,dy\\
&\lesssim \|\Omega\|_{L^\infty}2^{-k(n+1)}\int_{\mathbb{R}^{n}}2^{-(k-N(j)) n}\varphi\Big(\frac{x-y}{2^{k-N(j)}}\Big)  d y\\
&\lesssim \|\Omega\|_{L^\infty}2^{-k(n+1)}. %\chi_{|x|\le 2^{k+2}}.
\end{align*}

Combining  the estimates in the  two cases above, we have that
\begin{align}\label{b0}
&\sum_{k\in \Bbb Z}|K_{k}\ast \varphi_{k-N(j)}(x-y)|\\
&\lesssim \|\Omega\|_{L^\infty}\sum_{k\in \Bbb Z}\frac{ 2^{k}}{|x-y|^{n+2}}1_{|x-y|>2^{k+2}}(x-y)+ \|\Omega\|_{L^\infty}\sum_{k\in \Bbb Z}2^{-k(n+1)}1_{|x-y|\le 2^{k+2}}(x-y)\nonumber\\
&\simeq \frac{\|\Omega\|_{L^\infty}}{|x-y|^{n+1}}\sum_{k\in \Bbb Z}\frac{ 2^{k}}{|x-y|}1_{|x-y|>2^{k+2}}(x-y)+\frac{\|\Omega\|_{L^\infty}}{|x-y|^{n+1}}\nonumber\\&\lesssim\frac{\|\Omega\|_{L^\infty}}{|x-y|^{n+1}},\nonumber
\end{align} and
\begin{align}\label{b2}\sum_{k\in \Bbb Z}|(b(x)-b(y))K_{k}\ast \varphi_{k-N(j)}(x-y)|&
\lesssim\|\nabla b\|_{L^\infty} \frac{\|\Omega\|_{L^\infty}}{|x-y|^n}.
\end{align}
Then we get \begin{align}\label{K10}|K_1(x,y)|&\lesssim \|\nabla b\|_{L^\infty} \frac{\|\Omega\|_{L^\infty}}{|x-y|^n}.\end{align}

We now consider $\nabla (K_{k}\ast \varphi_{k-N(j)})(x).$
Again,  we consider the following  two cases $|x|>2^{k+2}$ and $|x|\le 2^{k+2}.$

Case 1:  $|x|>2^{k+2}$.

Since $|x-y|\ge |x|-|y|\ge |x|-\frac{|x|}{2}\ge \frac{|x|}{2},$ we have
\begin{align*}
|\nabla (K_{k}\ast \varphi_{k-N(j)})(x)| &=| K_{k}\ast \nabla\varphi_{k-N(j)}(x)|\\
&=\bigg|\int_{\mathbb{R}^{n}} \frac{\Omega(y^{\prime})}{|y|^{n+1}} 1_{2^{k}<|y|<2^{k+1}}(y) 2^{N(j)-k}(\nabla\varphi)_{k-N(j)}(x-y)d y\bigg|\\
&\lesssim \int_{\mathbb{R}^{n}} \frac{|\Omega(y^{\prime})|}{|y|^{n+1}} 1_{2^{k}<|y|<2^{k+1}}(y) 2^{N(j)-k}\frac{2^{3(k-N(j))}}{(2^{k-N(j)}+|x-y|)^{n+3}}d y\\
&\lesssim \frac{2^{2(k-N(j))}}{|x|^{n+3}}\int_{\mathbb{R}^{n}} \frac{|\Omega(y^{\prime})|}{|y|^{n+1}} 1_{2^{k}<|y|<2^{k+1}} (y)d y  \\
&\lesssim \|\Omega\|_{L^1} \frac{2^{-2N(j)}2^k}{|x|^{n+3}}.%\chi_{|x|>2^{k+2}}
\end{align*}

Case 2:  $|x|\le 2^{k+1}$.

We get
\begin{align*}|\nabla (K_{k}\ast \varphi_{k-N(j)})(x)|
&=\bigg|\int_{\mathbb{R}^{n}} \frac{\Omega(y^{\prime})}{|y|^{n+1}} 1_{2^{k}<|y|<2^{k+1}}(y) 2^{N(j)-k}(\nabla\varphi)_{k-N(j)}(x-y)d y\bigg|\\
&\lesssim \int_{\mathbb{R}^{n}} \frac{|\Omega(y^{\prime})|}{|y|^{n+1}} 1_{2^{k}<|y|<2^{k+1}}(y) 2^{N(j)-k}|(\nabla\varphi)_{k-N(j)}(x-y)|d y\\
&\lesssim 2^{N(j)}\frac{\|\Omega\|_{L^\infty}}{2^{k(n+2)}}\int_{\mathbb{R}^{n}}|(\nabla\varphi)_{k-N(j)}(x-y)| d y\\ &\lesssim2^{N(j)}\frac{\|\Omega\|_{L^\infty}}{2^{k(n+2)}}. %\chi_{|x|\le2^{k+2}}.
\end{align*}

Thus, for any $x,y,z$ satisfying $2|y-z|\le |x-y|$ (which gives $|x-y|\simeq|x-z|$), we have
\begin{align} \label{b3}
&\sum_{k\in \Bbb Z}|K_{k}\ast \varphi_{k-N(j)}(x-y)-K_{k}\ast \varphi_{k-N(j)}(x-z)|\\
&\lesssim\sum_{k\in \Bbb Z}|\nabla K_{k}\ast \varphi_{k-N(j)}((1-\theta)(x-y)+\theta(x-z))||y-z| \nonumber\\&\lesssim\|\Omega\|_{L^\infty}\bigg(\sum_{k\in \Bbb Z}\frac{2^k}{|x-y|^{n+3}}1_{|x-y|>2^{k+2}}(x-y)+2^{N(j)}\sum_{k\in \Bbb Z}\frac{1}{2^{k(n+2)}}1_{|x-y|\le2^{k+2}}(x-y)\bigg)|y-z| \nonumber\\
&\lesssim\|\Omega\|_{L^\infty}\bigg(\frac{1}{|x-y|^{n+2}}\sum_{k\in \Bbb Z}\frac{2^k}{|x-y|}1_{|x-y|>2^{k+2}}(x-y)+2^{N(j)}\frac{1}{|x-y|^{n+2}}\bigg)|y-z|\nonumber
\\
&\lesssim\|\Omega\|_{L^\infty}\bigg(\frac{1}{|x-y|^{n+2}}+2^{N(j)}\frac{1}{|x-y|^{n+2}}\bigg)|y-z|\nonumber
\\
&\lesssim\|\Omega\|_{L^\infty}2^{N(j)}\frac{|y-z|}{|x-y|^{n+2}}.\nonumber
\end{align}

By using the above inequality and $\eqref{b0}$, we have that for any $x,y,z$ satisfying $2|y-z|\le |x-y|$,
%then $|x-y|\simeq|x-z|$ we get
\begin{align} \label{b4}
&\sum_{k\in \Bbb Z} \Big|(b(x)-b(y))K_{k}\ast \varphi_{k-N(j)}(x-y)-(b(x)-b(z))K_{k}\ast \varphi_{k-N(j)}(x-z)\Big|\\
&\le|b(x)-b(y)|\sum_{k\in \Bbb Z}|K_{k}\ast \varphi_{k-N(j)}(x-y)-K_{k}\ast \varphi_{k-N(j)}(x-z)|\nonumber\\
&\quad+|b(y)-b(z)|\sum_{k\in \Bbb Z}|K_{k}\ast \varphi_{k-N(j)}(x-z)| \nonumber\\
&\lesssim\|\Omega\|_{L^\infty}\|\nabla b\|_{L^\infty}\Big(2^{N(j)}\frac{|y-z|}{|x-y|^{n+1}}+\frac{|y-z|}{|x-z|^{n+1}}\Big ) \nonumber\\
&\lesssim\|\Omega\|_{L^\infty}\|\nabla b\|_{L^\infty}2^{N(j)}\frac{|y-z|}{|x-y|^{n+1}}. \nonumber
\end{align}
Combined \eqref{b2} and \eqref{b4},
for any $x,y,z$ satisfying $2|y-z|\leq |x-y|,$ we get
\begin{align*}
\sum_{k\in \Bbb Z}|(b(x)-b(y))K_{k}\ast \varphi_{k-N(j)}(x-y)-(b(x)-b(z))K_{k}\ast \varphi_{k-N(j)}(x-z)|\lesssim\frac{\omega_{1,j}\Big(\frac{|y-z|}{|x-y|}\Big)}{|x-y|^n},
\end{align*}where
\begin{align*}
\omega_{1,j}(t) =\|\Omega\|_{L^{\infty}} \|\nabla b\|_{L^\infty}\min (1,2^{N(j)} t).
\end{align*}
The Dini norm of this function $\omega_{1,j}$ is given by
\begin{align}\label{dini norm omega j}
 \int_{0}^{1} \omega_{1,j}(t) \frac{d t}{t} &\lesssim \|\Omega\|_{L^{\infty}} \|\nabla b\|_{L^\infty}\bigg(\int_{0}^{2^{-N(j)}} 2^{N(j)} t \frac{d t}{t}+\int_{2^{-N(j)}}^{1} \frac{d t}{t}\bigg) \\
 &=\|\Omega\|_{L^{\infty}} \|\nabla b\|_{L^\infty}(1+\log 2^{N(j)})\nonumber\\
 & \lesssim\|\Omega\|_{L^{\infty}} \|\nabla b\|_{L^\infty}(1+N(j)).\nonumber
 \end{align}
Thus, we obtain that for any $x,y,z$ satisfying $2|y-z|\le |x-y|,$
\begin{align}|K_1(y,x)-K_1(z,x)|+|K_1(x,y)-K_1(x,z)|\lesssim\frac{\omega_{1,j}\Big(\frac{|y-z|}{|x-y|}\Big)}{|x-y|^n}.\end{align}

The proof of Lemma  \ref{lemtj} is complete. \qed

\quad\quad

\subsection { Proof of \eqref{vv1} and \eqref{vq31}} To begin with,  we need the following lemma which gives \eqref{vv1}.

\begin{lemma}\label{v21} Let  $b\in Lip({\Bbb R}^n)$ and $\Omega\in L^1({\Bbb S}^{n-1})$ and satisfies \eqref{mean zero}. The following inequality holds for  some $\tau\in (0,1),$\begin{align}\label{frourier 2}
\|[b, T_{2,\,j}^{N}] f\|_{L^{2}}
& \lesssim \|\Omega\|_{L^{1}}\|\nabla b\|_{L^\infty} 2^{-\tau N(j-1)}\|f\|_{L^{2}}.
\end{align}
\end{lemma}

\emph{ Proof.} Recall that

 \begin{align*}[b, T_{2,\,j}^{N}]f&=\sum_{k\in \Bbb Z}[b, T_k(S_{k+N(j-1)}-S_{k+N(j)})] f, \end{align*}
where $$S_{k+N(j-1)}-S_{k+N(j)}=\sum_{i=N(j-1)+1}^{N(j)}(S_{k+i-1}-S_{k+i})=\sum_{i=N(j-1)+1}^{N(j)}\Delta_{k+i}^3.$$
Then for $1<p<\infty$
\begin{align*}\|[b, T_{2,\,j}^{N}] f\|_{L^{2}}&\le \sum_{i=N(j-1)+1}^{N(j)}\Big\|\sum_{k\in \Bbb Z}[b, T_k\Delta_{k+i}^3] f\Big\|_{L^{2}}. \end{align*}
Now, we give the dedicated $L^2$-norm of $\sum_{k\in \Bbb Z}[b, T_k\Delta_{k+i}^3] f.$ For any $i\in \Bbb N,\, k\in \mathbb{Z}$, denote by  $\widetilde{T}_k^if(x)=T_{k}\Delta_{k+i}f(x).$
 We may write
$$\begin{array}{cl}[b,T_{k}\Delta_{k+i}^3]f
&=[b,\Delta_{k+i}](\widetilde{T}_{k}^i\Delta_{k+i}f)+\Delta_{k+i}([b,\widetilde{T}_{k}^i]\Delta_{k+i}f)
+\Delta_{k+i}\widetilde{T}_{k}^i([b,\Delta_{k+i}]f).
\end{array}$$
Then we get
\begin{align*}
\Big\|\sum_{k\in \Bbb Z}[b,T_{k}\Delta_{k+i}^3]f\Big\|_{L^{2}}
& =\Big\|\sum_{k\in \Bbb Z}[b, \Delta_{k+i}]\widetilde{T}_k^i\Delta_{k+i} f\Big\|_{L^{2}}+\Big\|\sum_{k\in \Bbb Z} \Delta_{k+i}[b,\widetilde{T}_k^i]\Delta_{k+i}f\Big\|_{L^{2}}+\Big\|\sum_{k\in \Bbb Z} \Delta_{k+i}\widetilde{T}_k^i[b,\Delta_{k+i}]f\Big\|_{L^{2}}\\
&=:I\!I\!I_1+I\!I\!I_2+I\!I\!I_3.
\end{align*}
Denote
\begin{align*}
\widehat{{\widetilde{T}}_{k}^i f}(\xi)= \widehat{K}_{k}(\xi) \widehat{\psi}(2^{k+i} \xi) \widehat{f}(\xi)=: \widetilde{m}_{i,k}(\xi) \widehat{f}(\xi).
\end{align*}
We now consider the pointwise estimates for $\widetilde{m}_{i,k}(\xi)$ and its derivatives.
By using the cancellation condition of $\Omega$ as in \eqref{mean zero}, it is easy to verify that
$$|\widehat{K}_{k}(\xi)| \lesssim 2^{-k}|2^k\xi|^2\|\Omega\|_{L^1} .$$
Since $\widehat{\psi} \in \mathcal{S}(\mathbb{R}^{n})$ with  ${\rm supp}\widehat{\psi}\subset \{1/2\le |\xi|\leq 2\}$, we see that $|2^k\xi|\approx 2^{-i}$. Hence,
\begin{align}\label{m3}
|\widetilde{m}_{i,k}(\xi)|\lesssim2^{-k}2^{-2i}\|\Omega\|_{L^1} .
\end{align}
and for multi-index $\alpha$ with $|\alpha|=2,$
\begin{align}\label{m4}
|\partial^\alpha \widetilde{m}_{i,k}(\xi)|\lesssim2^{k}\|\Omega\|_{L^1}.
\end{align}

After combining \eqref{m3} and \eqref{m4}, we apply Lemma \ref{CD1} to $[b,{\widetilde{T}}_{k}^i]$ to obtain that there exists some constant
$0<\lambda<1$ such that
\begin{align}\label{t3}
\|[b,{\widetilde{T}}_{k}^i]f\|_{L^2}\lesssim2^{-2\lambda i}\|\Omega\|_{L^1}\|\nabla b\|_{L^\infty}  \|f\|_{L^2},\ \ \ i\ge 1.
\end{align}
Moreover, by using the Plancherel theorem and \eqref{m3}, we
get
\begin{align}\label{t4}\|{\widetilde{T}}_{k}^i f\|_{L^2}\lesssim 2^{-k}  2^{-2i}\|\Omega\|_{L^1}\|f\|_{L^ 2}.\end{align}
We now  estimate $I\!I\!I_1$, $I\!I\!I_2$ and $I\!I\!I_3$, respectively.
By using \eqref{t3} and the Littlewood--Paley theory, we get that
\begin{align*}
I\!I\!I_1&\lesssim \bigg(\dsum_{k\in \mathbb Z}\|[b,{\widetilde{T}}_{k}^i](\Delta_{k+i}f)\|_{L^2}^2\bigg)^{1/2}\\
&\lesssim 2^{-2\lambda i}\|\Omega\|_{L^1}\|\nabla b\|_{L^\infty} \bigg(\dsum_{k\in \mathbb Z}\|\Delta_{k+i}f\|_{L^2}^2\bigg)^{1/2}\\
&\lesssim2^{-2\lambda i
}\|\nabla b\|_{L^\infty}\|\Omega\|_{L^1}\|f\|_{L^2}.
\end{align*}
 Now, we estimate $I\!I\!I_2.$  By using \eqref{t4} and Lemma  \ref{CD0}, we get
 \begin{align*}
 I\!I\!I_2
&\lesssim \bigg(\dsum_{k\in \mathbb Z}\|{\widetilde{T}}_{k}^i([b,\Delta_{k+i}]f)\|_{L^2}^2\bigg)^{1/2}\\
&\lesssim2^{-i}\|\Omega\|_{L^1}\bigg(\dsum_{k\in \mathbb
Z}2^{-2(k+i)}\|[b,\Delta_{k+i}]f\|_{L^2}^2\bigg)^{1/2}\\
&\lesssim2^{-i}\|\nabla b\|_{L^\infty}\|\Omega\|_{L^1}\|f\|_{L^2}.\end{align*}
Finally,
 by duality and  the estimate of $I\!I\!I_1$, we obtain that
 \begin{align*}
I\!I\!I_3
&\lesssim 2^{-i}\|\nabla b\|_{L^\infty}\|\Omega\|_{L^1}\|f\|_{L^2}.
\end{align*}
It follows from $I\!I\!I_1,\,I\!I\!I_2$ and $I\!I\!I_3$ that there exists some constant $0<\gamma<1,$
$$\Big\|\sum_{k\in \Bbb Z}[b, T_k\Delta_{k+i}^3] f\Big\|_{L^{2}}\lesssim2^{-\gamma i}\|\nabla b\|_{L^\infty}\|\Omega\|_{L^1}\|f\|_{L^2},\quad \hbox{for}\,\,i\ge 1.$$
Then for $\tau\in (0,1)$\begin{align*}
\|[b, T_{2,\,j}^{N}]f\|_{L^{2}}&\le\sum_{i=N(j-1)+1}^{N(j)}\Big\|\sum_{k\in \Bbb Z}[b, T_k\Delta_{k+i}^3] f\Big\|_{L^{2}}\\
&\lesssim\sum_{i=N(j-1)+1}^{N(j)}2^{- \gamma i}\|\nabla b\|_{L^\infty}\|\Omega\|_{L^1}\|f\|_{L^2}\\&\lesssim2^{- \tau N(j-1)}\|\nabla b\|_{L^\infty}\|\Omega\|_{L^1}\|f\|_{L^2}.
 \end{align*}
 
 The proof of Lemma \ref{v21} is complete.
 \qed

\quad\quad
\begin{lemma}\label{lemtj1} Let  $b\in Lip({\Bbb R}^n)$. If $\Omega\in L^1({\Bbb S}^{n-1})$ and satisfies \eqref{mean zero}, then the operator $[b, T_{2,\,j}^{N}]$  is a  Calder\'{o}n-Zygmund operator with kernel $K_2(x,y)$ satisfies

\begin{align*}
|K_2(x,y)|\lesssim \frac{\|\Omega\|_{L^1}\|\nabla b\|_{L^\infty}}{|x-y|^n}
 \end{align*}
 and for $2|h|\le |x-y|,$
 \begin{align*}
|K_2(x,y+h)-K_2(x,y)|+|K_2(x,y)-K_2(x+h,y)|\lesssim\frac{\omega_{2,j}(\frac{|h|}{|x-y|})}{|x-y|^{n}},
 \end{align*}
 where $\omega_{2,j}(t)=\|\Omega\|_{L^1}\|\nabla b\|_{L^\infty}\min\{1,2^{N(j)}t\}.$

\end{lemma}
\quad\quad

In fact, applying   Lemma \ref{lemtj1}  and Lemma \ref{lem 4} to $T=[b, T_{2,\,j}^{N}]$, we get that for $1<p<\infty$ and $w\in A_p$\begin{align}\label{i}
\|[b, T_{2,\,j}^{N}]f\|_{L^p(w)}&\lesssim \|\Omega\|_{L^{1}}\|\nabla b\|_{L^\infty}(1+N(j)){\{w\}}_{A_p}\|f\|_{L^p(w)},
\end{align} which gives the proof of  \eqref{vq31}. Taking $w=1,$ we get that for $1<p<\infty,$
\begin{align}\label{ti}
\|[b, T_{2,\,j}^{N}]f\|_{L^p}&\lesssim \|\Omega\|_{L^{\infty}}\|\nabla b\|_{L^\infty}(1+N(j))\|f\|_{L^p}.
\end{align}
Interpolating \eqref{ti} and \eqref{frourier 2}, we get \begin{align*}
\|[b, T_{2,\,j}^{N}]f\|_{L^p}&\lesssim \|\Omega\|_{L^{1}}\|\nabla b\|_{L^\infty}(1+N(j))2^{-\vartheta N(j-1)}\|f\|_{L^p}.
\end{align*} which gives \eqref{vv1}.

\quad\quad

\emph{\textbf{Proof of Lemma \ref{lemtj1}}.}
Recall the definition of $[b,{T}_{2,j}^{N}]$ given as by

\begin{align*}
[b, T_{2,\,j}^{N}]f&=\sum_{k\in \Bbb Z}[b, T_k(S_{k+N(j-1)}-S_{k+N(j)})] f.
\end{align*}

Now, we study size and smoothness of the kernel $K_2(x,y).$ Again, it suffices to verify the pointwise bound for $|K_{k}\ast \varphi_{k+N(j)}(x)|$ and $|\nabla K_{k}\ast \varphi_{k+N(j)}(x)|.$ Note that
\begin{align*}
|K_{k}\ast \varphi_{k+N(j)}(x)| &=\Big|\int_{\mathbb{R}^{n}} \frac{\Omega(y^{\prime})}{|y|^{n}} 1_{2^{k}<|y|<2^{k+1}} 2^{-(k+N(j)) n} \varphi(\frac{x-y}{2^{k+N(j)}}) d y\Big|.
\end{align*}
We consider the following two cases:  $|x|>2^{k+2}$ and $|x|\le 2^{k+2}.$

Case 1: $|x|>2^{k+2}$.

In this case we get $|x-\theta y|\ge |x|-\theta |y|\ge |x|-\frac{|x|}{2}\ge \frac{|x|}{2}$ for any $\theta\in (0,1).$

%Since $$\int_{{\Bbb S}^{n-1}}\Omega(y')(y')^\gamma\,d\sigma(y')=0,\,\,|\gamma|\le 1.$$
% expanding $\varphi_{k+N(j)}(x)$ into a Taylor series
%around $x$ gives that
By using the cancellation condition of $\Omega$ as in \eqref{mean zero} and the Taylor's expansion of $\varphi_{k+N(j)}(x-y)$ around $y$, we get that
\begin{align*}
|K_{k}\ast \varphi_{k+N(j)}(x)| &\le\sum_{|\alpha|=2}\int_{\mathbb{R}^{n}} \frac{|\Omega(y^{\prime})|}{|y|^{n+1}} 1_{2^{k}<|y|<2^{k+1}}(y) |\partial^\alpha\varphi_{k+N(j)}(x-\theta y)||y|^2 \,d y\\
&=\int_{\mathbb{R}^{n}} \frac{|\Omega(y^{\prime})|}{|y|^{n+1}} 1_{2^{k}<|y|<2^{k+1}}(y)\frac{1}{2^{2(k+N(j))}} \frac{2^{2(k+N(j))}|y|^2}{(2^{k+N(j)}+|x-\theta y|)^{n+2}} d y\\
&\lesssim\frac{2^k}{|x|^{n+2}}\int_{\mathbb{R}^{n}} \frac{|\Omega(y^{\prime})|}{|y|^{n}} 1_{2^{k}<|y|<2^{k+1}}(y) d y \\
&\lesssim \|\Omega\|_{L^1}\frac{ 2^{k}}{|x|^{n+2}}. %\chi_{|x|>2^{k+2}}.
\end{align*}

Case 2:   $|x|\le 2^{k+2}$.

We get
\begin{align*}
|K_{k}\ast \varphi_{k+N(j)}(x)| &\le\int_{\mathbb{R}^{n}} \frac{|\Omega(y^{\prime})|}{|y|^{n+1}} 1_{2^{k}<|y|<2^{k+1}}(y) 2^{-(k+N(j)) n}(y) \varphi(\frac{x-y}{2^{k+N(j)}})\,dy\\
&\lesssim 2^{-(k+N(j)) n}\int_{\mathbb{R}^{n}} \frac{|\Omega(y^{\prime})|}{|y|^{n+1}} 1_{2^{k}<|y|<2^{k+1}}(y) \,dy\\
&\lesssim \|\Omega\|_{L^1}2^{-k(n+1)}. % \chi_{|x|\le 2^{k+2}}.
\end{align*}
Combining  the above two cases, we have that
\begin{align}\label{kb1}
&\sum_{k\in \Bbb Z}|K_{k}\ast \varphi_{k+N(j)}(x-y)|\\
&\lesssim \|\Omega\|_{L^1}\sum_{k\in \Bbb Z}\frac{ 2^{k}}{|x-y|^{n+2}}1_{|x-y|>2^{k+2}}(x-y)+\|\Omega\|_{L^1} \sum_{k\in \Bbb Z}2^{-k(n+1)}1_{|x-y|\le 2^{k+2}}(x-y)\nonumber\\
&\lesssim \frac{\|\Omega\|_{L^1}}{|x-y|^{n+1}},\nonumber
\end{align}
and hence,
\begin{align}\label{kb2}
\sum_{k\in \Bbb Z}|(b(x)-b(y))K_{k}\ast \varphi_{k+N(j)}(x-y)|&\lesssim \frac{\|\Omega\|_{L^1}\|\nabla b\|_{L^\infty}}{|x-y|^n}.
\end{align}
This gives
\begin{align}\label{kbr}
|K_2(x,y)|&\lesssim \frac{\|\Omega\|_{L^1}\|\nabla b\|_{L^\infty}}{|x-y|^n}.
\end{align}

We now estimate $\nabla K_{k}\ast \varphi_{k+N(j)}(x).$ We also consider the following two cases: $|x|>2^{k+2}$ and $|x|\le 2^{k+2}.$

Case 1: $|x|>2^{k+2}$.

We get $|x- \theta y|\ge |x|- \theta |y|\ge |x|-\frac{|x|}{2}\ge \frac{|x|}{2},\,0<\theta<1$.
By using the cancellation condition of $\Omega$,
 %$\int_{{\Bbb S}^{n-1}}\Omega(y')\,d\sigma(y')=0,$ then
 we get that
\begin{align*}
|\nabla (K_{k}\ast \varphi_{k+N(j)})(x)| &=| K_{k}\ast \nabla\varphi_{k+N(j)}(x)|\\
&=\bigg|\int_{\mathbb{R}^{n}} \frac{\Omega(y^{\prime})}{|y|^{n+1}} 1_{2^{k}<|y|<2^{k+1}}(y) 2^{-N(j)-k}(\nabla\varphi)_{k+N(j)}(x-y)d y\bigg|\\
&=\bigg|\int_{\mathbb{R}^{n}} \frac{\Omega(y^{\prime})}{|y|^{n+1}} 1_{2^{k}<|y|<2^{k+1}}(y) 2^{-N(j)-k}\bigg[(\nabla\varphi)_{k+N(j)}(x-y)-(\nabla\varphi)_{k+N(j)}(x)\bigg]d y\bigg|\\
&\lesssim \int_{\mathbb{R}^{n}} \frac{|\Omega(y^{\prime})|}{|y|^{n+1}} 1_{2^{k}<|y|<2^{k+1}}(y) 2^{-2N(j)-2k}\frac{2^{3(k+N(j))}|y|}{(2^{k+N(j)}+|x-\theta y|)^{n+3}}d y\\
&\lesssim \frac{2^{k+N(j)}}{|x|^{n+3}}\int_{\mathbb{R}^{n}} \frac{|\Omega(y^{\prime})|}{|y|^{n}} 1_{2^{k}<|y|<2^{k+1}}(y) dy\\
&\lesssim \|\Omega\|_{L^1}\frac{2^{N(j)}2^k}{|x|^{n+3}}.%\chi_{|x|>2^{k+2}}.
\end{align*}

Case 2: $|x|\le 2^{k+2}$.

We have
\begin{align*}
|\nabla (K_{k}\ast \varphi_{k+N(j)})(x)| &=| K_{k}\ast \nabla\varphi_{k+N(j)}(x)|\\
&=\bigg|\int_{\mathbb{R}^{n}} \frac{\Omega(y^{\prime})}{|y|^{n+1}} 1_{2^{k}<|y|<2^{k+1}}(y) 2^{-N(j)-k}(\nabla\varphi)_{k+N(j)}(x-y)d y\bigg|\\
&\lesssim \int_{\mathbb{R}^{n}} \frac{|\Omega(y^{\prime})|}{|y|^{n+1}} 1_{2^{k}<|y|<2^{k+1}}(y) 2^{-N(j)-k}|(\nabla\varphi)_{k+N(j)}(x-y)|d y\\
&\lesssim 2^{(-N(j)-k)(n+1)}\int_{\mathbb{R}^{n}} \frac{|\Omega(y^{\prime})|}{|y|^{n+1}} 1_{2^{k}<|y|<2^{k+1}}(y) d y \\
&\lesssim \frac{\|\Omega\|_{L^1}}{2^{k(n+2)}}.
\end{align*}
Thus, for any $x,y,z$ satisfying $2|y-z|\le |x-y|,$  %then $|x-y|\simeq|x-z|$
\begin{align} \label{kb3}
&\sum_{k\in \Bbb Z}|K_{k}\ast \varphi_{k-N(j)}(x-y)-K_{k}\ast \varphi_{k-N(j)}(x-z)|\\
&\lesssim\sum_{k\in \Bbb Z}|\nabla K_{k}\ast \varphi_{k-N(j)}((1-\theta)(x-y)+\theta(x-z))||y-z| \nonumber\\
&\lesssim\|\Omega\|_{L^1}\bigg(\sum_{k\in \Bbb Z}2^{N(j)}\frac{2^k}{|x-y|^{n+3}}1_{|x-y|>2^{k+2}}(x-y)+\sum_{k\in \Bbb Z}\frac{1}{2^{k(n+2)}}1_{|x-y|\le2^{k+2}}(x-y)\bigg)|y-z| \nonumber\\
&\lesssim\|\Omega\|_{L^1}\bigg(\frac{2^{N(j)}}{|x-y|^{n+2}}+\frac{1}{|x-y|^{n+2}}\bigg)|y-z|\nonumber
\\&\lesssim\|\Omega\|_{L^1}2^{N(j)}\frac{|y-z|}{|x-y|^{n+2}}.\nonumber
\end{align}

From the above inequality  and \eqref{kb1} we get that  for any $x,y,z$ satisfying $2|y-z|\le |x-y|,$
\begin{align} \label{kb4}
&\sum_{k\in \Bbb Z}|(b(x)-b(y))K_{k}\ast \varphi_{k+N(j)}(x-y)-(b(x)-b(z))K_{k}\ast \varphi_{k+N(j)}(x-z)|\\
&\lesssim|(b(x)-b(y))|\sum_{k\in \Bbb Z}|K_{k}\ast \varphi_{k+N(j)}(x-y)-K_{k}\ast \varphi_{k+N(j)}(x-z)|\nonumber\\&\quad+|(b(y)-b(z))|\sum_{k\in \Bbb Z}|K_{k}\ast \varphi_{k+N(j)}(x-z)| \nonumber\\
&\lesssim\|\Omega\|_{L^1}\|\nabla b\|_{L^\infty}2^{N(j)}\frac{|y-z|}{|x-y|^{n+1}}+\|\Omega\|_{L^1}\|\nabla b\|_{L^\infty}\frac{|y-z|}{|x-z|^{n+1}} \nonumber\\
&\lesssim2^{N(j)}\|\Omega\|_{L^1}\|\nabla b\|_{L^\infty}\frac{|y-z|}{|x-y|^{n+1}}. \nonumber
\end{align}

Combined \eqref{kb2} and \eqref{kb4}, we obtain that  for any $x,y,z$ satisfying $2|y-z|\le |x-y|,$
\begin{align*}\sum_{k\in \Bbb Z}|(b(x)-b(y))K_{k}\ast \varphi_{k+N(j)}(x-y)-(b(x)-b(z))K_{k}\ast \varphi_{k+N(j)}(x-z)|\lesssim\frac{\omega_{2,j}\Big(\frac{|y-z|}{|x-y|}\Big)}{|x-y|^n}.
\end{align*}where
\begin{align*}
\omega_{2,j}(t) = \|\Omega\|_{L^1}\|\nabla b\|_{L^\infty} \min (1,2^{N(j)} t).
\end{align*}
This yeilds that   for any $x,y,z$ satisfying $2|y-z|\le |x-y|,$
\begin{align*}|K_2(y,x)-K_{2}(z,x)|+|K_2(x,y)-K_{2}(x,z)|\lesssim\frac{\omega_{2,j}\Big(\frac{|y-z|}{|x-y|}\Big)}{|x-y|^n}.\end{align*}

The proof of Lemma \ref{lemtj1} is complete. \qed

\bigskip

\section{Proof of  Theorem \ref{thm2}}\label{s8}

  Write
$$\begin{array}{cl}
[b,T_\Omega]\nabla f(x)=-T_\Omega[b, \nabla]f(x)+[b, \nabla T_\Omega]f(x).\end{array}$$
For the first
term, since $$[b, \nabla]f=-(\nabla b)f,$$  applying  Theorem 1.4 in \cite{HRT} yields for $1<p<\infty$ and $w\in A_p,$
 \begin{align}\label{TOmega}
 \|T_\Omega[b, \nabla]f\|_{L^p(w)}&\lesssim{\{w\}}_{A_p}(w)_{A_p}\|\Omega\|_{L^\infty}\|(\nabla b) f\|_{L^p(w)}\\&\lesssim \nonumber {\{w\}}_{A_p}(w)_{A_p}\|\Omega\|_{L^\infty}\|\nabla b\|_{L^\infty}\|f\|_{L^p(w)}.
 \end{align}
 For the second term, write $K(x)=\frac{\Omega(x')}{|x|^n},$ the commutator  \begin{align*}[b, \nabla T_\Omega] f(x)=p.v.\dint_{{\Bbb R}^n}\nabla K(x-y)(b(x)-b(y) f(y)\,dy.\end{align*}  It is easy to verify that  $\nabla K(x)$ is homogeneous of degree $-n-1.$ A trivial computation gives that \begin{align*}\|\nabla K\|_{L^\infty(\Bbb S^{n-1})}\lesssim \|\Omega\|_{L^\infty}+\|\nabla \Omega\|_{L^\infty}\end{align*} and \begin{align*}(x_k\nabla K(x))^{\wedge}(\xi)&=i\xi_k\widehat{\nabla K}(\xi)=i\dfrac{\partial}{\partial \xi_k}(i\xi_1\widehat{K}(\xi),\cdots,i\xi_n\widehat{K}(\xi)).\end{align*}
Moreover,
$$
\dfrac{\partial}{\partial \xi_k}(\xi_j\widehat{K})(\xi)=\begin{cases}\widehat{ K}(\xi)+\xi_j\dfrac{\partial\widehat{K}(\xi)}{\partial \xi_k}\ &\ \text{as}
j=k;\\
\xi_j\dfrac{\partial\widehat{K}(\xi)}{\partial \xi_k}\ &\ \text{as}\ j\not=k.
\end{cases}
$$
So,
$$(x_k\nabla K(x))^{\wedge}(0)=0\ \ \forall\ \ k\in\{1,\cdots, n\}.
$$
Additionally,
$\widehat{ \nabla K}(\xi)=i \xi\widehat{ K}(\xi)$, then $\widehat{\nabla K}(0)=0.$

This says  $$\int_{\mathbb S^{n-1}}(x_k')^\gamma\nabla K(x')\,d\sigma(x')=0\ \ \forall\ \ k\in\{1,\cdots,n\}, \gamma\in\{0,1\}.
$$
Since $|\nabla K(x')|\in L^\infty(\mathbb S^{n-1}),$
by  using Theorem \ref{thm1}, we see that  \begin{align}\label{C}\|[b, \nabla T_\Omega]f\|_{L^p(w)}\lesssim{\{w\}}_{A_p}(w)_{A_p}\|\nabla K\|_{L^\infty({\Bbb S}^{n-1})}\|\nabla b\|_{L^\infty}\|f\|_{L^p(w)}.\end{align}
Combining the estimates for \eqref{TOmega} and \eqref{C}, we get
 \begin{align}\nonumber\|[b,T_\Omega]\nabla f\|_{L^p(w)}&\lesssim{\{w\}}_{A_p}(w)_{A_p}(\|\Omega\|_{L^\infty}+\|\nabla \Omega\|_{L^\infty})\|\nabla b\|_{L^\infty}\|f\|_{L^p(w)},\end{align}
thereby reaching the first part of Theorem \ref{thm2}.

Moreover, regarding $\nabla[b,T_\Omega] f$ we have
$$\begin{array}{cl}
\nabla[b,T_\Omega] f(x)&=-[b, \nabla]T_\Omega f(x)+[b, \nabla T_\Omega]f(x)=-(\nabla b)(x)T_\Omega f(x)+[b, \nabla T_\Omega]f(x).\end{array}$$
In a similar way, we obtain
\begin{align}\|\nabla[b,T_\Omega] \|_{L^p(w)}&\lesssim \nonumber {\{w\}}_{A_p}(w)_{A_p}(\|\Omega\|_{L^\infty}+\|\nabla \Omega\|_{L^\infty})\|\nabla b\|_{L^\infty}\|f\|_{L^p(w)}.\end{align}
The proof of Theorem \ref{thm2} is complete.
\qed


\begin{thebibliography}{120}











\bibitem{C3} A. P. Calder\'on, \emph{Uniqueness in the Cauchy problem for partial differential equations}, Amer. J. Math. {\textbf{80}}  (1958), 16-36.

\bibitem{C1} A. P. Calder\'on,
\emph{Commutators of singular integrals,} Proc. Nat. Acad. Sci.
USA. \textbf{53} (1965), 1092-1099.



\bibitem{C2} A. P. Calder\'on,
\emph{Commutators, singular integrals on Lipschitz curves and application,} Proc. Inter. Con. Math.
Helsinki, 1978, 86-95.




\bibitem {CD1} Y. Chen, Y. Ding, G. Hong and J. Xiao, \emph{Some jump and variation inequalities for the Calder\'{o}n commutators and related operators,} arXiv:1709.03127.

\bibitem{CDH1}Y. Chen, Y. Ding and G. Hong, \emph{Commutators with fractional differentiation and new
characterizations of BMO-Sobolev spaces,} Anal. PDE. \textbf{9} (2016), 1497-1522.







\bibitem{Co} J. Cohen,\emph{ A sharp estimate for a multilinear singular integral in  $\mathbb{ R}^n,$} Indiana Univ. Math. J. \textbf{30 }(1981), 693-702.




\bibitem{CM1} R. Coifman and Y. Meyer, \emph{On commutators of singular integrals and bilinear integrals}, Trans. Amer. Math.
Soc. \textbf{212} (1975), 315-331.

\bibitem{CM} R. Coifman and Y. Meyer, \emph{Au dela des operateurs pseudo-differentiels,} Asterisque, no.57, Soc. Math. de France, 1978.\label{CM}








\bibitem{CCDO} J.M. Conde-Alonso, A. Culiuc, F. Di Plinio and Y. Ou, \emph{A sparse domination
principle for rough singular integrals}, Anal. PDE.\textbf{ 10} (2017), no. 5, 1255-1284.







\bibitem{DHL}
F. Di Plinio, T.P. Hyt\"onen and K. Li
\emph{Sparse bounds for maximal rough singular integrals via the Fourier transform}, Annales de l'institut Fourier, to appear.
Available at
https://arxiv.org/abs/1706.09064.








\bibitem{Fu}N. Fujii, \emph{Weighted bounded mean oscillation and singular integrals,} Math. Japon. \textbf{22}
(1977/1978), 529-534.

























	


















\bibitem{DR86}
J. Duoandikoetxea and J. Rubio de Francia, \emph{Maximal and
singular integral operators via Fourier transform estimates,}
Invent. Math. \textbf{84} (1986), 541-561.




\bibitem{Fe} C. Fefferman,
\emph{Recent Progress in classical Fourier analysis,} Proc. Inter. Con. Math.
Vancouver, 	1974, 95-118.







\bibitem{GH}  L. Grafakos and P. Honz\'{i}k, \emph{A weak-type estimate for commutators}, Int. Math. Res. Not.
\textbf{2012} (2012), 4785-4796.





\bibitem{SH} S. Hofmann, \textit{ Weighted inequalities for commutators of rough  singular integrals,
}   Indian Univ. Math. J. \textbf{ 39} (1990), 1275-1304.


















\bibitem{HRT} T. Hyt\"{o}nen, L. Roncal and O. Tapiola, \emph{Quantitative weighted estimates for
rough homogeneous singular integrals,} Israel J. Math. \textbf{218} (2017), no. 1, 133-
164.



\bibitem{Lac15} M. T. Lacey, \emph{An elementary proof of the $A_2$ bound},  Israel J. Math. \textbf{217}(2017), 181-195.



\bibitem{Ler 0} A.K. Lerner, \emph{On pointwise estimates involving sparse operators,} New York J. Math. \textbf{22} (2017), 341-349.



\bibitem{Ler 2} A. K. Lerner, \emph{A note on weighted bounds for rough singular integrals,} C. R. Acad. Sci. Paris, Ser. I. \textbf{356} (2018), no. 1, 77-80.








\bibitem {Me} Y. Meyer, \emph{Ondelettes et Op\'erateurs}, Vol. II, Hermann, Paris, 1990. \label{Me}
	
\bibitem {MC} Y. Meyer and R. Coifman, \emph{Ondelettes et Op\'erateurs,} Vol. III, Hermann, Paris, 1991.\label{MC} 	




\bibitem{MuWh71} B. Muckenhoupt, R. L. Wheeden, \emph{Weighted norm inequalities for singular and fractional integrals}, Tran. Amer. Math. Soc. \textbf{161} (1971), 249-258.

\bibitem{Mu} T. Murai, \emph{Boundedness of singular integral operators of Calder\'{o}n-type}, Adv.   Math.
\textbf{59} (1986), 71-81.





\bibitem{CM2} C. Muscalu, \emph{Calder\'{o}n commutators and the Cauchy integral on Lipschitz curves revisited II. The Cauchy integral and its generalizations,}
Revista Mat. Iberoam. \textbf{30} (2014),  1089-1122.



\bibitem{sw} E. M. Stein and G. Weiss,\emph{ Interpolation of operators with change of measures,} Trans. Amer.
Math. Soc. \textbf{87 }(1958), 159-172.








\bibitem{Ta3} M. Taylor, \emph{Commutator estimates for H\"{o}lder continuous and bmo-Sobolev
multipliers}, Proc. Amer. Math. Soc. \textbf{143} (2015), 5265-5274.

\bibitem{Wi} J. M. Wilson, \emph{Weighted inequalities for the dyadic square function without dyadic $A_1$}, Duke
Math. J. \textbf{55} (1987), 19-50.


\bibitem{Y} A. Youssfi, \emph{Regularity properties of commutators and BMO-Triebel-Lizorkin spaces},  Ann. Inst. Fourier (Grenoble) \textbf{45} (1995), 795-807.




 \end {thebibliography}
\end{document}